\documentclass[11pt,twoside]{amsart}

\usepackage{xspace}
\usepackage{amssymb}
\usepackage{bbm}
\usepackage{graphicx}
\usepackage{url}
\usepackage{pifont}

\numberwithin{equation}{section}

\newcommand{\mi}{\bbi\xspace}

\DeclareMathSymbol{\varnothing}{\mathord}{AMSb}{"3F}

\DeclareMathOperator{\im}{Im}

\DeclareMathOperator*{\res}{Res}

\DeclareMathOperator{\tr}{tr}

\newcommand{\bbC}{\mathbb{C}}

\newcommand{\bbN}{\mathbb{N}}

\newcommand{\bbP}{\mathbb{P}}

\newcommand{\bbR}{\mathbb{R}}

\newcommand{\bbZ}{\mathbb{Z}}

\theoremstyle{plain}
\newtheorem{theorem}{Theorem}[section]
\newtheorem*{theorem*}{Theorem}
\newtheorem{corollary}[theorem]{Corollary}
\newtheorem*{corollary*}{Corollary}
\newtheorem{proposition}[theorem]{Proposition}
\newtheorem*{proposition*}{Proposition}
\newtheorem{lemma}[theorem]{Lemma}
\newtheorem*{lemma*}{Lemma}
\newtheorem{example}[theorem]{Example}
\newtheorem*{example*}{Example}
\newtheorem{definition}[theorem]{Definition}
\newtheorem*{definition*}{Definition}

\newtheorem*{notation*}{Notation}

\newtheorem*{remark*}{Remark}

\numberwithin{figure}{section}

\newcommand{\bbi}{\mathbbm{i}}

\newlength{\PWIDTH}
\newlength{\PSPACE}
\title{Dressing curves}
\author{M.~Kilian}
\address{M. Kilian, Department of Mathematics,
University College Cork, Ireland.}
\email{m.kilian@ucc.ie}


\begin{document}

\begin{abstract}
We present a loop group description for curves in $\mathbb{R}^3$, and apply it to classify the circletons: Circles dressed by simple factors.
\end{abstract}


\maketitle

\section*{Introduction}
Zakharov and Shabat \cite{ZakS2} introduced a method called dressing which provides a construction of new solutions from old ones in the theory of integrable equations. In particular, dressing by specific very simple maps corresponds to the classical transformations already discovered by 19th century geometers like Darboux, B\"acklund and Bianchi. In the modern theory of integrable systems these simple maps are called simple factors \cite{TerU}.  

Consider a smooth closed curve in $\mathbb{R}^3$. It neccessarily has periodic curvature and torsion functions, but conversely periodicity of curvature and torsion is not suffiecient to yield a closed curve. One method to obtain a wealth of examples of closed curves is to deform a given closed curve, and keeping it closed during the deformation, as the isoperiodic deformation by Grinevich and Schmidt \cite{GriS1, GriS-curves}. Calini and Ivey \cite{CalI2} use this deformation to deform multiply wrapped circles and study the resulting knot type. While the isoperiodic deformation is a continuous deformation, dressing by simple factors is not, and constructs closed curves that are not slight pertubations of the original curve. Calini and Ivey \cite{calini-ivey1998} present a B\"acklund transformation for curves which is very different to our approach via loop groups.

Starting with a multiply wrapped circle we investigate what happens when it is dressed by suitably chosen simple factors that preserve periodicity. It turns out that for a circle with wrapping number $\omega$ there are precisly $\omega -1$ simple factors which do this. In analogy to the solitons we call the resulting closed curves circletons. We classify the single-circletons: curves obtained by dressing mutiply wrapped circles by one simple factor. This then immediately also leads to the classification of multi-circletons: curves obtained by dressing  mutiply wrapped circles by products of simple factors. 

%
%

\section{Extended frames}

Let $q: \bbR \to \bbC$ be a smooth map, and set 
\[
	 V(q) = \frac{1}{2}\,\begin{pmatrix} \mi \lambda & q \\ -\bar{q} & -\mi\lambda \end{pmatrix}
\]
Here $\lambda \in \bbC$ is an additional parameter. Let $\dot{}$ denote differentiation with respect to $t \in \mathbb{R}$. The unique solution to the inital value problem 
\[
	\dot{F}_\lambda = F_\lambda\,V(q)\,,\qquad F_\lambda (0) = \mathbbm{1}
\]
gives for each $\lambda$ a map $F_\lambda :\bbR \to \mathrm{SL}_2(\bbC)$. It is entire in $\lambda$ and smooth in $t$. Since $V(q) \in \mathfrak{su}_2$ for all $\lambda \in \bbR$, it follows that $F_\lambda \in \mathrm{SU}_2$ for all $\lambda \in \bbR$. Suppose $\kappa,\,\tau:\mathbb{R} \to \mathbb{R}$ are smooth functions. Hashimoto \cite{hasimoto1972} defines the {\emph{complex curvature function}}
\[
	q(t) = \kappa (t) \exp [\,\mi \int_0^t \tau (s)\,ds\,] \,.
\]
Indeed a straightforward computation shows that 
\[
	\gamma (t) = 2 \left. F'_\lambda \right|_{\lambda =0} \,F^{-1}_0
\]
is an arc-length parameterized curve in $\mathfrak{su}_2 \cong \bbR^3$ with curvature $\kappa$ and torsion $\tau$. Here $\prime$ denotes differentiation with respect to $\lambda$. The map $F_\lambda$ is called an {\emph{extended frame}} of the curve $\gamma$. Under the identification  $\bbR^3 \cong \mathfrak{su}_2$, we have
\[
	| X | = \sqrt{ \det X}\,,\quad \langle X,\,Y \rangle = -\tfrac{1}{2} \tr (X\,Y)\,,\quad \mbox{and} \quad X \times Y = \tfrac{1}{2} [\,X,\,Y\,]\,.
\]
We sometimes need to pick an orthonormal basis of $\mathfrak{su}_2$, so set
\[
	\sigma_1 = \begin{pmatrix} 0 & 1 \\ -1 & 0 \end{pmatrix}\,, \quad
	\sigma_2 = \begin{pmatrix} 0 & \mi \\ \mi & 0 \end{pmatrix}\,, \quad
	\sigma_3 = \begin{pmatrix} \mi & 0 \\ 0 & -\mi \end{pmatrix} .
\]
Suppose $\kappa,\,\tau$ are periodic, so that $\kappa(t + \varrho) = \kappa(t)$ and $\tau(t + \varrho) = \tau(t)$ for all $t \in \bbR$ for some period $\varrho \in \mathbb{R}^\times$. Then for the {\emph{monodromy}} $M_\lambda (\varrho) = F_\lambda (\varrho)$ we have that 
\[
	\gamma(t + \varrho) = \gamma(t) \qquad \mbox{ for all } t \in \bbR
\]
is equivalent to 
\begin{equation*} 
	M_0 (\varrho) = \pm \mathbbm{1} \quad \mbox{and} \quad  \left. M'_\lambda (\varrho) \right|_{\lambda =0} = 0
\end{equation*}
%
%
\section{Birkhoff Factorization}
Let $r>0$ and define
\begin{align*}
	C_r &= \{ \lambda \in \bbC \mid \| \lambda \| = r \}\,, \\
	D_r &= \{ \lambda \in \bbC \mid \| \lambda \| < r \}\,, \\
	E_r &= \{ \lambda \in \bbC \mid \| \lambda \| > r \} \cup \{ \infty \} \,.
\end{align*}
Next, define the following $r$-loop groups of $\mathrm{SL}_2 (\mathbb{C})$ by 
\begin{align*}
	\Lambda_r  \mathrm{SL}_2 (\mathbb{C}) &= \{ \mbox{ Analytic maps } C_r \to \mathrm{SL}_2 (\mathbb{C}) \}\,, \\
	\Lambda_r^+  \mathrm{SL}_2 (\mathbb{C}) &= \{ \,g \in \Lambda_r \mathrm{SL}_2 (\mathbb{C}) \mid g \mbox{ extends analytically to } D_r  \}\,, \\
	\Lambda_r^-  \mathrm{SL}_2 (\mathbb{C}) &= \{ \,g \in \Lambda_r \mathrm{SL}_2 (\mathbb{C}) \mid g \mbox{ extends analytically to } E_r  \}\,, \\
	\Lambda_{r,\mathbbm{1}}^-  \mathrm{SL}_2 (\mathbb{C}) &= \{ \,g \in \Lambda_r^- \mathrm{SL}_2 (\mathbb{C}) \mid g(\infty) = \mathbbm{1} \}\,.
\end{align*}
For $g \in \Lambda_r \mathrm{SL}_2 (\mathbb{C})$ there is a factorization $g = g_+ D\,g_-$, where $g_- \in \Lambda_r^- \mathrm{SL}_2 (\mathbb{C}),\,g_+ \in \Lambda_r^+ \mathrm{SL}_2 (\mathbb{C})$, and $D = \mathrm{diag}[\,\lambda^n,\,\lambda^{-n}\,]$ for some $n \in \mathbb{Z}$. To obtain uniqueness of the factors in the Birkhoff factorization we impose the additional condition that $g_- \in \Lambda_{r,\mathbbm{1}}^-  \mathrm{SL}_2 (\mathbb{C})$. 

{\bf{Birkhoff Factorization}} \cite{PreS, McI} For every $g \in \Lambda_r  \mathrm{SL}_2 (\mathbb{C})$ there exists a unique factorization  
\begin{equation*} 
	g = g_+ D\,g_-
\end{equation*} 
with $g_- \in \Lambda_{r,\mathbbm{1}}^- \mathrm{SL}_2 (\mathbb{C}),\,g_+ \in \Lambda_r^+ \mathrm{SL}_2 (\mathbb{C})$, and $D = \mathrm{diag}[\,\lambda^n,\,\lambda^{-n}\,]$ for some $n \in \mathbb{Z}$ is a diagonal matrix. When $n = 0$, so that $D = \mathbbm{1}$, then $g$ is said to lie in the \emph{big cell} of $\Lambda_r \mathrm{SL}_2 (\mathbb{C})$.

Since the domains of the above $r$-loop groups are invariant under complex conjugation it makes sense to consider $g(\bar{\lambda})$. Define
\[
	g^*(\lambda) = \overline{g(\bar{\lambda})^t}^{\,-1}.
\]
The {\emph{reality condition}} $g(\lambda) \in \mathrm{SU}_2$ for all real $\lambda$ in the domain of $g$ is then conveniently written as 
\begin{equation*} 
	g^* = g\,.
\end{equation*}
We need notation for the subgroups whose elements satisfy this reality condition, and denote
\begin{align*}
	\Lambda_{r,*}  \mathrm{SL}_2 (\mathbb{C}) &= \{ g \in \Lambda_{r}  \mathrm{SL}_2 (\mathbb{C})  \mid g^* = g \}\,, \\
	\Lambda_{r,*}^+  \mathrm{SL}_2 (\mathbb{C}) &= \{ \,g \in \Lambda_r^+ \mathrm{SL}_2 (\mathbb{C}) \mid  g^* = g \}\,, \\
	\Lambda_{r,*}^-  \mathrm{SL}_2 (\mathbb{C}) &= \{ \,g \in \Lambda_r^- \mathrm{SL}_2 (\mathbb{C}) \mid g^* = g \}\,, \\
	\Lambda_{r,*\mathbbm{1}}^-  \mathrm{SL}_2 (\mathbb{C}) &= \{ \,g \in \Lambda_{r,\mathbbm{1}}^- \mathrm{SL}_2 (\mathbb{C}) \mid g^* = g \}\,.
\end{align*}
\begin{lemma} Let $g \in \Lambda_{r,*} \mathrm{SL}_2 (\mathbb{C})$. Then $g$ lies in the big cell, and has a unique factorization $g = g_+g_-$ with  $g_- \in \Lambda_{r,*\mathbbm{1}}^- \mathrm{SL}_2 (\mathbb{C})$ and $g_+ \in \Lambda_{r,*}^+ \mathrm{SL}_2 (\mathbb{C})$. 
\end{lemma}
\begin{proof}
Suppose $g$ does not lie in the big cell, and $D(\lambda) = \mathrm{diag} [\,\lambda^n,\,\lambda^{-n}\,]$ with $n \neq 0$. Then $g_+^* D^{-1} g_-^* =g_+ D g_-$. Rearranging this gives 
\[
g_+^{-1} g_+^* D^{-1} = D g_- g_-^{*^{-1}}
\]
Set $G_+ = g_+^{-1}g_+^*$ and $G_- = g_-g_-^{*-1}$. Clearly $G_+ \in \Lambda_r^+ \mathrm{SL}_2 (\mathbb{C})$ and $G_- \in \Lambda_{r,\mathbbm{1}}^- \mathrm{SL}_2 (\mathbb{C})$. Since the diagonal entries of $G_-$ are positive and real for $\lambda \in \bbR \cap E_r$, we conclude that also $G_+$ has non-zero diagonal entries. If $n<0$, then the entries of the first column of $G_+ D^{-1}$  are holomorphic functions at $\lambda = 0$, while the first row of $DG_-$ are holomorphic at $\lambda = \infty$. In particular for the first diagonal entries we get the relation
\[
	\lambda^{-n} G_+^{(11)} = \lambda^{n} G_-^{(11)}\,.
\]
This implies that $G_+^{(11)} = G_-^{(11)} \equiv 0$, giving the contradiction. 
The case $n >0$ is proven similarly by considering the second diagonal entries. 
Thus $g$ lies in the big cell and has a unique Birkhoff factorization $g = g_+ g_-$. By the reality condition we also have the Birkhoff factorization $g = g_+^* g_-^*$, so by uniqueness of the Birkhoff factorization we conclude $g_+ = g_+^*,\,g_- = g_-^*$.
\end{proof}
\begin{corollary} Let $g \in \Lambda_{r,*\mathbbm{1}}^-  \mathrm{SL}_2 (\mathbb{C})$ and $h \in \Lambda_{r,*}^+  \mathrm{SL}_2 (\mathbb{C})$. Then both $gh$ and $hg$ lie in the big cell of $\Lambda_r \mathrm{SL}_2 (\mathbb{C})$.
\end{corollary}
%
%

\section{Dressing}
For $g \in \Lambda_r^-  \mathrm{SL}_2 (\mathbb{C})$ and  $h \in  \Lambda_r^+  \mathrm{SL}_2 (\mathbb{C})$ define 
\[
	g \# h = (g\,h)_+ \,.
\]
This defines an action $\Lambda_r^-  \mathrm{SL}_2 (\mathbb{C}) \times \Lambda_r^+  \mathrm{SL}_2 (\mathbb{C}) \to \Lambda_r^+  \mathrm{SL}_2 (\mathbb{C})$ called the {\emph{dressing action}}.
\subsection{Potentials}
Consider meromorphic maps on $E_r$ which only have one pole, namely a simple pole at $\lambda = \infty$, and take values in $\mathfrak{sl}_2 (\mathbb{C})$, take values in $\mathfrak{su}_2$ along $\lambda \in \mathbb{R}$ and the first two terms in the Laurent expansion are restricted:
\[
	\Lambda_{r,*}^{-\infty,1} \mathfrak{sl}_2 (\mathbb{C}) = \left\{ \xi (\lambda) =  \sum_{j\leq 1} \xi_j \lambda^{j} \mbox{ with } \xi_j \in \mathfrak{su}_2 \mbox{ for all } j,\,\xi_1 = \tfrac{1}{2} \sigma_3,\, \xi_0 \perp \xi_1  \right\} .
\] 
Smooth 1-forms on $\mathbb{R}$ of the form
\[
	\eta(t,\,\lambda) = \tfrac{1}{2}\sigma_3 \lambda dt + \xi_0 (t) dt +   	 
\sum_{j\leq -1} \xi_j (t) \lambda^{j} dt \mbox{ with } \xi_j:\mathbb{R} \to \mathfrak{su}_2 \mbox{ smooth, and }  \xi_0(t) \perp \sigma_3 \,\forall t \in \mathbb{R}
\]
will be called {\emph{potentials}}. The proof of the next proposition follows immediately by comparing coefficients of series expansions.
\begin{proposition} \label{th:curve-dpw}
Let $\eta$ be a potential. Let $\phi_0 \in \Lambda_{r,*\mathbbm{1}}^-  \mathrm{SL}_2 (\mathbb{C})$, and 
$\phi :\bbR \times  E_r \setminus \{\infty\} \to \mathrm{SL}_2(\mathbb{C})$ the unique solution of the initial value problem
\[
	d\phi = \phi \eta\,,\quad \phi(0,\,\lambda) = \phi_0\,.
\]
Then $\phi$ is in the big cell for all $t \in \mathbb{R}$, so has for each $t \in \mathbb{R}$ a unique Birkhoff factorization
\[
	\phi(t,\,\lambda) = \phi_+(t,\,\lambda) \,\phi_-(t,\,\lambda)
\]
with smooth maps $\phi_+: \mathbb{R} \to \Lambda^+_{r,*} \mathrm{SL}_2 (\mathbb{C})$,  $\phi_-: \mathbb{R} \to \Lambda^-_{r,*\mathbbm{1}} \mathrm{SL}_2 (\mathbb{C})$, and
\[
	\phi_+^{-1} \dot{\phi}_+ = \frac{1}{2}\,\begin{pmatrix} \mi \lambda & f(t,\,\lambda) \\ -f^*(t,\,\lambda) & -\mi\lambda \end{pmatrix}
\]
for some smooth map $f:\mathbb{R} \times  I_r \to \mathbb{C}$. Further, the curve $\gamma: \mathbb{R} \to \mathfrak{su}_2$ defined by 
\[
	\gamma (t) = 2 \left. \phi_+' \,\phi_+^{-1}  \right|_{\lambda =0} 
\]
is an arc-length parameterized curve with complex curvature $ f(t,\,0)$.
\end{proposition}
As a consequence of the last proposition the pair $(\eta,\,\phi_0 )$ uniquely determines a curve. The converse can be shown as in \cite{DorPW}, although the uniqueness fails without further normalizations. A curve is said to be of {\emph{finite type}}, if it arises from a constant potential with finitely many terms. 
We call these {\emph{finite type potentials}}, and denote them by $\xi dt$ with $\xi = \xi(\lambda)$ an element of
\[
	\Lambda_{r,*}^{-g,1} \mathfrak{sl}_2 (\mathbb{C}) = \bigl\{ \xi \in  \Lambda_{r,*}^{-\infty,1} \mathfrak{sl}_2 (\mathbb{C}) \mid \xi (\lambda) =  \sum_{j=-g}^1 \xi_j \lambda^{j}  \bigr\}
\]  
for some $g \in \mathbb{N} \cup \{ 0\}$. 
%
%
\subsection{Symes map} For a finite type potential the solution of $d\phi = \phi \xi dt,\,\phi_0 = \mathbbm{1}$ is $\phi = \exp \int \xi dt$. Combining this with a subsequent Birkhoff splitting gives a map $\Phi: \xi \mapsto \phi_+$, called the {\emph{Symes map}}  
\[
	\Phi(\xi) = \bigl(\, \exp \int \xi dt \,\bigr)_+\,.
\]
The next result is the analogy of Proposition~2.7 in \cite{BurP:dre}.
\begin{lemma} \label{th:symes}
For $g \in \Lambda_r^-  \mathrm{SL}_2 (\mathbb{C})$ and finite type potential $\xi dt$ we have $\Phi (g\,\xi\, g^{-1} ) = g \# \bigl( \Phi (\xi)\,g^{-1}(\infty) \bigr)$. In particular, if $g \in \Lambda_{r,\mathbbm{1}}^-  \mathrm{SL}_2 (\mathbb{C})$ then $$\Phi (g\,\xi\, g^{-1} ) = g \# \Phi (\xi).$$
\end{lemma}
\begin{proof}
\begin{align*}
	\Phi (g\,\xi\, g^{-1} ) &=   \bigl(  g\,\exp\int\xi  \,g^{-1} \bigr)_+ 
	=  \bigl(  g\,\exp\int\xi  \,g^{-1}(\infty) \tilde{g} \bigr)_+ \\ &=  \bigl(  g\,\exp\int\xi  \,g^{-1}(\infty) \bigr)_+ \quad \mbox{since }\tilde{g} = g(\infty) g^{-1} \in \Lambda_{r,\mathbbm{1}}^- \mathrm{SL}_2(\bbC),\\
&= \bigl(  g\,\Phi (\xi) (\exp\int\xi )_- \,g^{-1}(\infty) \bigr)_+ \\
	&= \bigl(  g\,\Phi (\xi) \,g^{-1}(\infty) \bigr)_+  \quad \mbox{since } g(\infty) (\exp\int\xi )_- \,g^{-1}(\infty) \in \Lambda_{r,\mathbbm{1}}^- \mathrm{SL}_2(\bbC) \\
	&= g \# \bigl( \Phi (\xi)\,g^{-1}(\infty) \bigr)\,.
\end{align*}
\end{proof}
%
%
\subsection{Spectral curve} 

Let $\xi \in \Lambda_{r,*}^{-g,1} \mathfrak{sl}_2 (\mathbb{C})$. Then $\det \xi$ is a finite Laurent series with real coefficients of the form
\begin{equation} \label{eq:finite_series}
	\det \xi = \tfrac{1}{4} \lambda^2 + \sum_{i=0}^{2g} a_i \lambda^{-i} \quad \mbox{ with } a_i \in \mathbb{R}\,.
\end{equation}
Consider $\xi \in \Lambda_{r,*}^{-g,1} \mathfrak{sl}_2 (\mathbb{C})$ whose $2g + 2$ roots of $\det\xi$ are pairwise distinct, and define 
\[
	\Sigma^* = \{ (\nu,\,\lambda) \in \mathbb{C}^2 \mid \det (\nu\mathbbm{1} - \xi(\lambda)) = 0 \}.
\]
Thus $\Sigma^*$ consists of all solutions of 
\begin{equation} \label{eq:nu}
	\nu^2 = -\det \xi (\lambda)\,.
\end{equation}
By construction this gives a degree 2  map $\lambda : \Sigma^* \to \mathbb{C}^\times$, which is branched at the $2g + 2$ distinct roots of $\det \xi$. Compactifying the affine curve defined by \eqref{eq:nu} by adding in two points $\infty^\pm$ over the point $\lambda = \infty$ gives a hyperelliptic Riemann surface $\Sigma$ of genus $g$, called the spectral curve. It has the hyperelliptic involution $\sigma(\nu,\,\lambda) = (-\nu,\,\lambda)$, as well as the anti-holomorphic involution $\rho (\nu,\,\lambda) = (-\bar{\nu},\,\bar{\lambda})$ fixing the set $\lambda \in \mathbb{R}$. An arc-length parameterized curve $\gamma$  in $\mathbb{R}^3$ is called a {\emph{finite type curve}} if it comes from a finite type potential in the sense of Proposition \ref{th:curve-dpw}. Furthermore we then call $\Sigma$ the spectral curve of $\gamma$. The number $g \in \mathbb{N} \cup \{0\}$ is called the {\emph{spectral genus}}. 
\begin{example}
Let us briefly consider the case $g=0$, since all computations are elementary. Then $ \Lambda_{r,*}^{0,1} \mathfrak{sl}_2 (\mathbb{C})$ consists 
of elements of the form
\[
	\xi (\lambda) = \frac{1}{2} \begin{pmatrix} \mi \lambda & u+\mi v\\ -u + \mi v & -\mi \lambda \end{pmatrix} \quad \mbox{ with } u,\,v \in \mathbb{R}\,.
\]
Then $\Phi(\xi) = \exp (t\,\xi(\lambda)) \in \Lambda^+_{r,*} \mathrm{SL}_2 (\mathbb{C})$ for any $r \in (0,\,\infty)$, and the arc-length parmetrized curve is
\[
	\gamma (t) = \begin{pmatrix}
\frac{\mi \sin \left(2 t \sqrt{u^2+v^2}\right)}{2 \sqrt{u^2+v^2}} & -\frac{\mi \sin ^2\left(t \sqrt{u^2+v^2}\right)}{u-\mi v} \\
 \frac{\sin ^2\left(t \sqrt{u^2+v^2}\right)}{\mi u-v} & -\frac{\mi \sin \left(2 t \sqrt{u^2+v^2}\right)}{2 \sqrt{u^2+v^2}}
\end{pmatrix},
\]
which is a circle unless $u=v=0$, in which case it is a line. Thus spectral genus $g=0$ curves are circles.
\end{example}
The spectral curve alone does not uniquely determine the arc-length parameterized curve, since there is a compact degree of freedom arising from the isospctral set, which we discuss next. 

\subsection{Isospectral dressing}

Suppose $a = a(\lambda)$ is a finite Laurent series of the form \eqref{eq:finite_series}. Then the isospectral set is defined as 
\[
	I(a) = \{ \xi \in   \Lambda_{r,*}^{-g,1} \mathfrak{sl}_2 (\mathbb{C}) \mid \det \xi = a \}\,.
\]
The {\emph{isospectral action}} $\Pi: \mathbb{R}^g \times I(a) \to I(a)$ is defined as follows: Let $\vec{x} = (x_0,\,\ldots,\,x_{g-1}) \in \mathbb{R}^g$ and write
\[
	\exp \bigl( \xi \sum_{i=0}^{g-1} \lambda^i x_i \bigr) = \phi_+ (\vec{x}) \,\phi_- (\vec{x})
\]
for the Birkhoff factoization. Define the map $\Pi$ by setting
\[
	\Pi(\vec{x})\,\xi = \phi_- (\vec{x})\, \xi \,\phi_- ^{-1}(\vec{x})\,.
\]
Since $\phi_+ (\vec{x}) \,\phi_- (\vec{x})$ commutes with $\xi$ we have that 
\[
	\Pi(\vec{x})\,\xi = \phi_- (\vec{x})\, \xi \,\phi_- ^{-1}(\vec{x}) = \phi_+^{-1} (\vec{x})\, \xi \,\phi_+(\vec{x})\,.
\]
The methods of \cite{hks1} can be easily modified to show that $\Pi$ defines a commutative transitive action on $I(a)$. Using this, Lemma~\ref{th:symes} and the above example yields a proof of the next
\begin{proposition} 
Every closed finite type curve is in the dressing orbit of a circle.
\end{proposition} 
%

%
%

\section{Simple factors}

We shall determine the simplest non-trivial elements of $\Lambda_{r,*\mathbbm{1}}^-  \mathrm{SL}_2 (\mathbb{C})$: Consider a degree 1 polynomial of the form 
\[
	Q(\lambda) = \mathbbm{1} + \lambda^{-1} P \quad \mbox{with } P \in \mathfrak{sl}_2(\bbC)
\]
Clearly $Q(\infty) = \mathbbm{1}$, and $\det Q(\lambda) = 1 + \lambda^{-2} \det P$. Normalizing the determinant to $1$, define
\[
	g(\lambda) = \frac{1}{\sqrt{1 + \lambda^{-2} |P|^2}} \bigl( \mathbbm{1} + \lambda^{-1} P \bigr)\,.
\]
Away from the roots of $\det Q$ we have $g^* = g$ if and only if $P \in \mathfrak{su}_2$. Thus $\det P >0$ if and only if $P \neq 0$, and then $\lambda = \pm \mi \sqrt{\det P} = \pm \mi |P|$ are the roots of $\det Q$. 
Then $g \in \Lambda_{r,*\mathbbm{1}}^- \mathrm{SL}_2(\bbC)$ for $r > |P|$, and  $g \in \Lambda_{r,*}^+ \mathrm{SL}_2(\bbC)$ for $r  < |P|$. Further, $g$ has eigenvalues
\[
	\mu_-(\lambda)  = \tfrac{\sqrt{\lambda^2 + |P|^2}}{\lambda + \mi |P|} \qquad \mu_+(\lambda) =  \tfrac{\sqrt{\lambda^2 + |P|^2}}{\lambda - \mi |P|}\,.
\]
Thus $\mu_-$ has only one pole, a simple pole at $\lambda = -\mi |P|$, and $\mu_+$ has only one pole, a simple pole at $\lambda = \mi |P|$.
Expand $P = y\sigma_1 + z \sigma_2 + x \sigma _3$. Note that $g$ has $\lambda$-independent eigenlines spanned by 
\[
	L_- = \begin{pmatrix} \frac{x- |P|}{\mi y + z} \\ 1\end{pmatrix}\,, \qquad
	L_+ = \begin{pmatrix} \frac{x+ |P|}{\mi y + z} \\ 1\end{pmatrix} .
\]
Then $\langle L_-,\,L_+ \rangle = \tfrac{2y}{y-\mi z}$. Since we want $g(\lambda) \in \mathrm{SU}_2$ for all $\lambda \in \bbR$, we set $y=0$. Thus $P =  z \sigma_2 + x \sigma _3$ . At $\lambda =  \pm \mi |P|$ the map $\mathbbm{1} + \lambda^{-1} P$ drops rank from 2 to 1, with kernels spanned respectively by 
\begin{equation*} 
	L = \begin{pmatrix} \frac{-z}{|P| + x}\\ 1\end{pmatrix}\,, \qquad 
	L^\perp = \begin{pmatrix} \frac{z}{|P| - x} \\ 1 \end{pmatrix}\,.
\end{equation*}
Let $\pi_L:\bbC^2 \to  \bbC \bbP^1$ be the hermitian projection onto $L \in \bbC \bbP^1$, and $\pi_L^\perp = \mathbbm{1}- \pi_L$. Set $\alpha = \mi |P|$, and consider 
\begin{equation*}
	\psi_{L,a} (\lambda) = (1- \alpha\lambda^{-1})\pi_L +  (1- \bar\alpha\lambda^{-1})\pi_L^\perp 
	= \mathbbm{1} + \alpha \lambda^{-1} (\pi_L^\perp  - \pi_L )
\end{equation*}
Then $\det \psi_{L,\alpha} = 1 - \alpha^2\lambda^{-2}$, and $\psi_{L,a}$ drops from rank 2 to 1 at $\pm\alpha$ with kernels 
\begin{equation*} 
	L = \mathrm{ker }\,\psi_{L,\alpha} (\alpha) \qquad L^\perp = \mathrm{ker }\,\psi_{L,\alpha} (\bar\alpha)\,.
\end{equation*}
A \emph{simple factor} is a map of the form $ (\det \psi_{L,\alpha})^{-1/2}\,\psi_{L,\alpha}$, and we denote these by 
  \begin{equation*} 
    h_{L,\alpha} = \frac{1}{\sqrt{1-\alpha^2\lambda^{-2}}} \bigl( \mathbbm{1} + \alpha \lambda^{-1} (\pi_L^\perp  - \pi_L ) \bigr)\,.
\end{equation*}
Note that 
  \begin{equation*} 
    h_{L,\alpha}^{-1} = \frac{1}{\sqrt{1-\alpha^2\lambda^{-2}}} \bigl( \mathbbm{1} - \alpha \lambda^{-1} (\pi_L^\perp  - \pi_L ) \bigr)\,.
\end{equation*}
The points $\pm \alpha$ are the singularities of the simple factor $h_{L,\alpha}$. In analogy to Proposition~4.2 in \cite{TerU} we have
\begin{lemma} \label{th:sf-dressing}
Let $h_{L,\alpha} \in \Lambda_{r,*\mathbbm{1}}^- \mathrm{SL}_2 (\bbC)$ be a simple factor, and $F_\lambda \in \Lambda_{r,*}^+ \mathrm{SL}_2 (\bbC)$. Then
\[
	h_{L,\alpha} \# F_\lambda =  h_{L,\alpha} \,F_\lambda\,h^{-1}_{L',\alpha} \qquad \mbox{where } L' = F_\alpha^{-1} \,L\,.
\]
\end{lemma} 
\begin{proof}
Clearly $h_{L,\alpha} \,F_\lambda\,h^{-1}_{L',\alpha}$ satisfies the reality condition, and is holomorphic in $D_r$ away from $\lambda = \pm \alpha$ where it has simple poles. The residues there are
\[
	\res_{\lambda = \alpha}  h_{L,\alpha} \,F\,h^{-1}_{L',\alpha} = 2\alpha \,\pi_L^\perp F_\alpha \pi_{L'} = 0
\]
since $\im (F_\alpha \pi_{L'} ) \subseteq L$. Similarly 
\[
	\res_{\lambda = -\alpha}  h_{L,\alpha} \,F\,h^{-1}_{L',\alpha} = -2\alpha \,\pi_L F_\alpha \pi_{L'}^\perp = 0
\]
since $\im (F_\alpha \pi_{L'}^\perp ) \subseteq L^\perp$. Hence $h_{L,\alpha} \,F_\lambda\,h^{-1}_{L',\alpha}$ is analytic in $D_r$, and thus 
$h_{L,\alpha} F_\lambda = (h_{L,\alpha} \,F_\lambda\,h^{-1}_{L',\alpha})\,(h_{L',\alpha})$ is the unique Birkhoff factorization.
\end{proof}
\begin{lemma}
Let $L \in \bbC\bbP^1$ and $M_\lambda (\varrho)$ the monodromy of an extended frame $F_\lambda$ such that at $\alpha \in \bbC \setminus \bbR$ we have
\[
	M^{-1}_\alpha (\varrho) L = L\,.
\]
Then the monodromy of $h_{L,\alpha} \# F_\lambda$ with respect to $\varrho$ is given by $h_{L,\alpha} M_\lambda (\varrho) h_{L,\alpha}^{-1}$.
\end{lemma}
\begin{proof} The monodromy $M_\lambda (\varrho)$ of $F_\lambda$ with respect to $\varrho$ is defined as $\varrho^* F_\lambda = M_\lambda (\varrho)\,F_\lambda$, where  $\varrho^* F_\lambda (t) = F_\lambda (t + \varrho)$. By Lemma~\ref{th:sf-dressing} we have 
$h_{L,\alpha} \# F_\lambda =  h_{L,\alpha} \,F_\lambda\,h^{-1}_{L',\alpha}$ with $L' = F_\alpha^{-1} \,L$. Then
\[
	\varrho^* h_{L',\alpha} = h_{\varrho^* L',\alpha} = h_{\varrho^* F_\alpha^{-1}L,\alpha} 
	= h_{ F_\alpha^{-1} M_\alpha^{-1}(\varrho) L,\alpha} = h_{ F_\alpha^{-1} L,\alpha} =  h_{L',\alpha}\,. 
\]
Thus $\varrho^* h_{L',\alpha} = h_{L',\alpha}$, and consequently 
\begin{align*}
	\varrho^* h_{L,\alpha} \# F_\lambda &=  h_{L,\alpha} \,\varrho^* F_\lambda\,\varrho^* h^{-1}_{L',\alpha} = h_{L,\alpha} \,\varrho^* F_\lambda\,h^{-1}_{L',\alpha} = h_{L,\alpha} \,M_\lambda (\varrho)\,F_\lambda h^{-1}_{L',\alpha} \\
&= h_{L,\alpha} \,M_\lambda (\varrho)\,h^{-1}_{L,\alpha} \,h_{L,\alpha}F_\lambda h^{-1}_{L',\alpha} = h_{L,\alpha} \,M_\lambda (\varrho)\,h^{-1}_{L,\alpha} \,\,h_{L,\alpha} \# F_\lambda\,.
\end{align*}
\end{proof}
Since the closing conditions for a curve in terms of the monodromy of the extended frame are preserved under conjugation, we have
\begin{corollary}
Suppose $\gamma$ is a $\varrho$-periodic curve with extended frame $F_\lambda$ and monodromy $M_\lambda (\varrho)$. Suppose there exists $\alpha \in \bbC \setminus \bbR$ and a line $L \in \bbC\bbP^1$ such that 
$M^{-1}_\alpha (\varrho) L = L$. Then the dressed curve with extended frame 
$h_{L,\alpha} \# F_\lambda$ is again $\varrho$-periodic.
\end{corollary}
%
%

\section{Circletons}

Extended frames for circles of radius $1/\kappa$ are given by 
\[
	F_\lambda (t) =  \cos \left(\tfrac{1}{2} t \sqrt{\kappa^2+\lambda^2}\,\right) \mathbbm{1} + \tfrac{\sin \left(\tfrac{1}{2} t \sqrt{\kappa^2+\lambda^2}\right)}{\sqrt{\kappa^2+\lambda^2}} \,V(\kappa) .
\]
Note that 
\[
	F_{\pm\mi\kappa} (t) =  \begin{pmatrix}
 1\mp\frac{t \kappa }{2} & \frac{t \kappa }{2} \\
 -\frac{t \kappa }{2} & 1 \pm \frac{t \kappa }{2}
\end{pmatrix}
\]
is not diagonalizable. A computation reveals that for $\varrho = 2 \pi$ and $M_\lambda (\varrho) = F_\lambda(2\pi)$ we have $M_0 (\varrho)  = -\mathbbm{1}$ and $M'_0  (\varrho)  = 0$, so we get a once-wrapped circle by restricting $t \in [0,\,2\pi]$. To obtain a $\omega$-wrapped circle for $\omega \in \bbN$ we let $t \in [0,\,2\omega\pi]$. The length of the curve determines the number of available resonance points. The eigenvalues of $F_\lambda$ are $\mu_\lambda (t) = \exp[ \,\pm \tfrac{\mi\,t}{2}\,\sqrt{\kappa^2 + \lambda^2}\,]$. The resonance points of a $\omega$-wrapped circle are all possible values of $\lambda$ at which $F_\lambda (2\omega\pi) = \pm \mathbbm{1}$. For the eigenvalues this means 
\[
	\mu_\lambda (2\omega\pi) = \pm 1 \,,	
\] 
excluding $\lambda = \pm \mi \kappa$, or equivalently $\omega \,\sqrt{\kappa^2 + \lambda^2} \in \bbZ^*$. We consider the case in $\kappa = 1$ in more detail. Then if $\omega \,\sqrt{\lambda^2 +1} = k$ for some $k \in \bbZ^*$, then 
\[
	\lambda = \pm \sqrt{ \frac{k^2}{\omega^2} -1}\,.
\]
The resonance points are candidates for the singularities of simple factors, so have to lie off the real line. Hence we require that $0 < k^2 < \omega^2$. This proves
\begin{proposition} Let $\gamma$ be an $\omega$-wrapped circle. Then there are precisely $\omega - 1$ many simple factors which dress $\gamma$ to a closed curve. In particular, an embedded circle cannot be dressed by a simple factor to a closed curve.
\end{proposition}
\begin{definition} A single-circleton is a closed curve obtained by dressing a multiply wrapped circle of radius 1 by one simple factor.
\end{definition} 
\begin{lemma}\label{th:circleton}
Up to isometry any single-circleton can be obtained by dressing a $\omega$-wrapped circle by a simple factor $h_{L,\alpha}$ with $L = [1:0]$, so of the form 
\[
	h_{L,\alpha} = \begin{pmatrix} \sqrt{\frac{\lambda -\alpha}{\lambda - \bar{\alpha}}} & 0 \\
	0 & \sqrt{\frac{\lambda - \bar{\alpha}}{\lambda -\alpha }}
	\end{pmatrix}\
\]
with purely imaginary $\alpha \in (0,\,\mi)$ given by 
\[
	\alpha = \sqrt{ \tfrac{k^2}{\omega^2} -1} \qquad \mbox{ with } \quad k \in \bbN,\, 1 \leq k < \omega\,.
\]
\end{lemma}
\begin{proof} Now $\mathrm{SU}_2$ acts transitively on $\bbC\bbP^1$ and $h_{UL,\alpha} = U h_{L,\alpha} U^{-1}$ for any $U \in \mathrm{SU}_2$. Dressing by  $h_{L,\alpha}$ and  $h_{UL,\alpha}$ gives the same curve up to isometry, so we may choose without loss of generality the line  $L = [1:0]$. Since $k \in \bbZ^*$ appears only squared, we can assume $k> 0$. Hence $k \geq 1$. Dressing by $ h_{L,\alpha}$ and $ h_{L,\bar\alpha}$  gives the same curve up to isometry, so it is no restriction to take $\alpha \in (0,\,\mi)$.  
\end{proof}
\begin{figure}[t] 
  	\includegraphics[width=1.85cm]{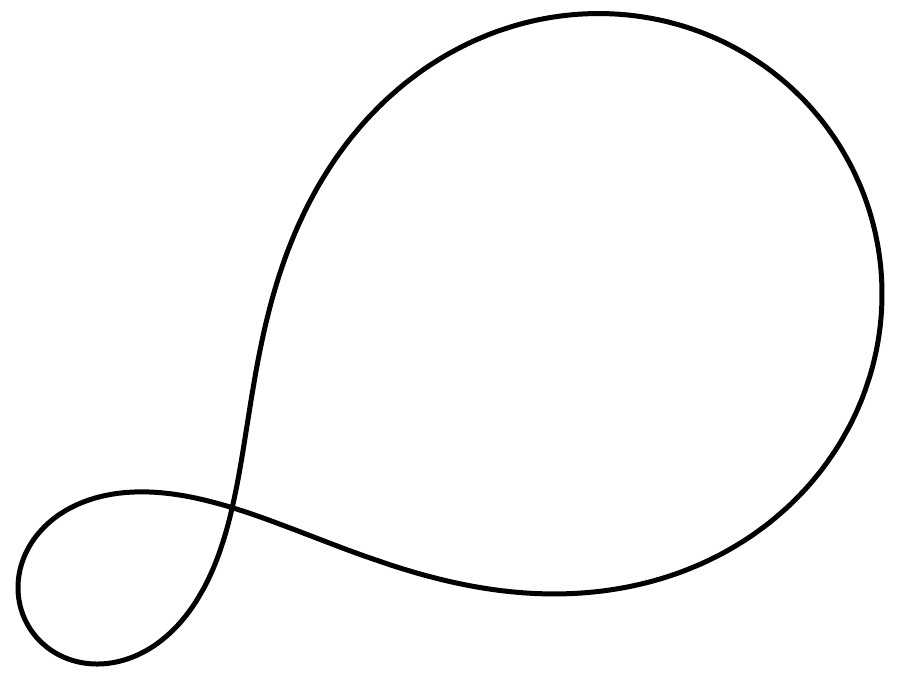}\\
  	\includegraphics[width=1.85cm]{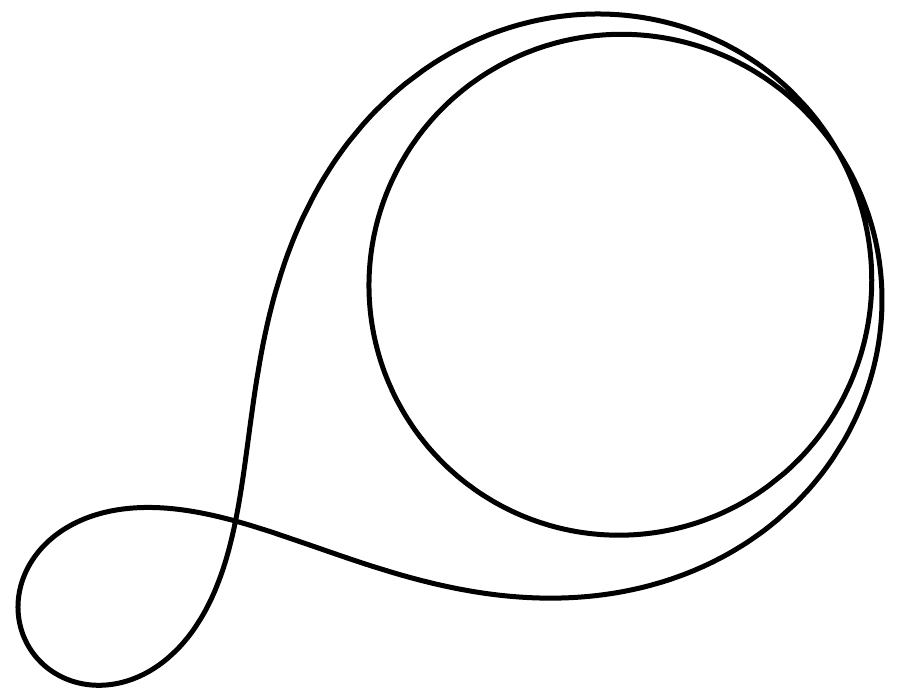}
  	\includegraphics[width=1.85cm]{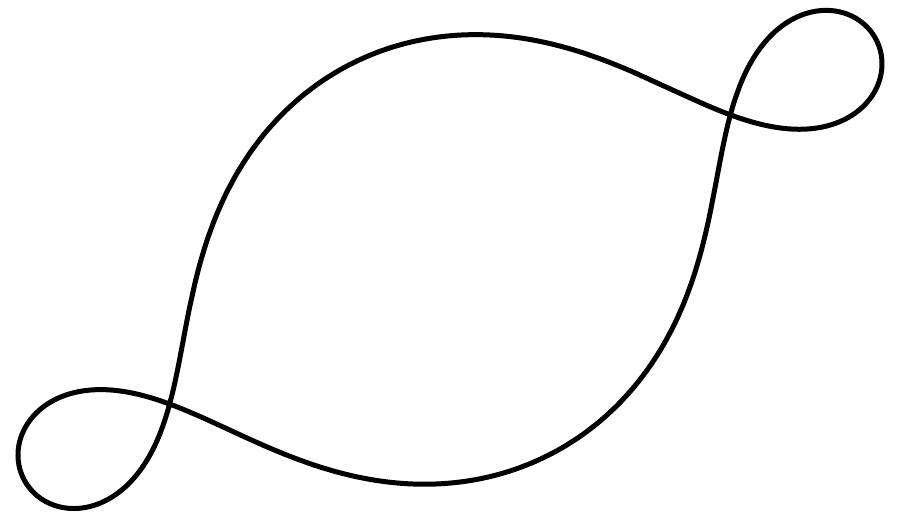}\\
	\includegraphics[width=1.85cm]{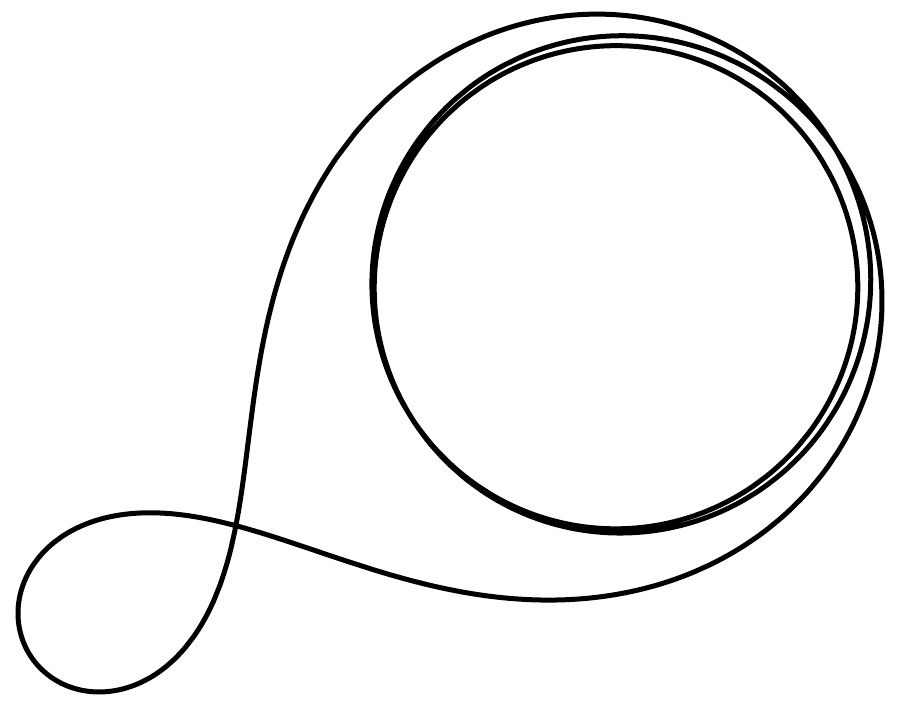}
  	\includegraphics[width=1.85cm]{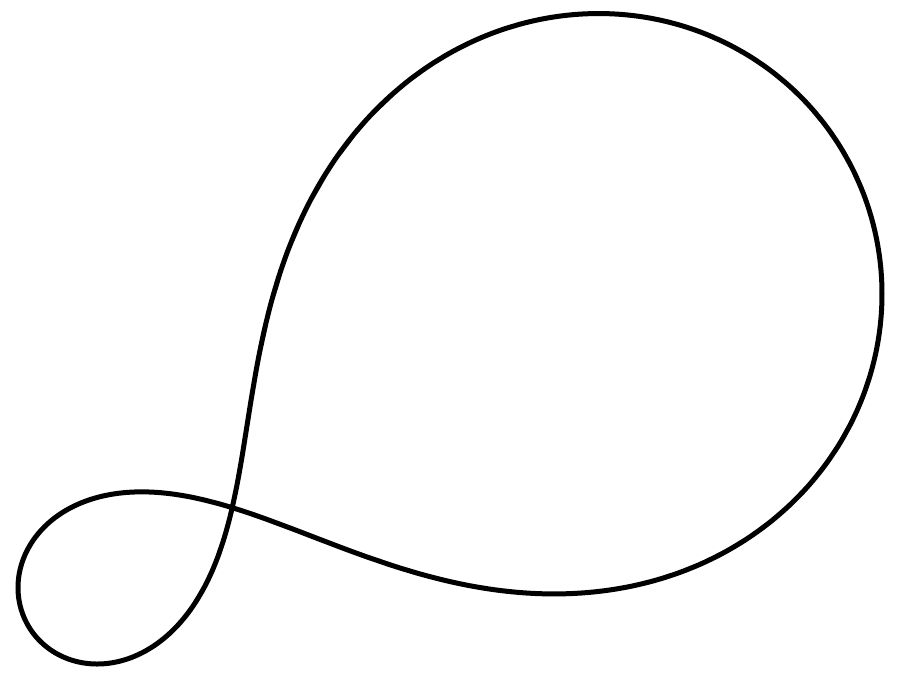}
  	\includegraphics[width=1.85cm]{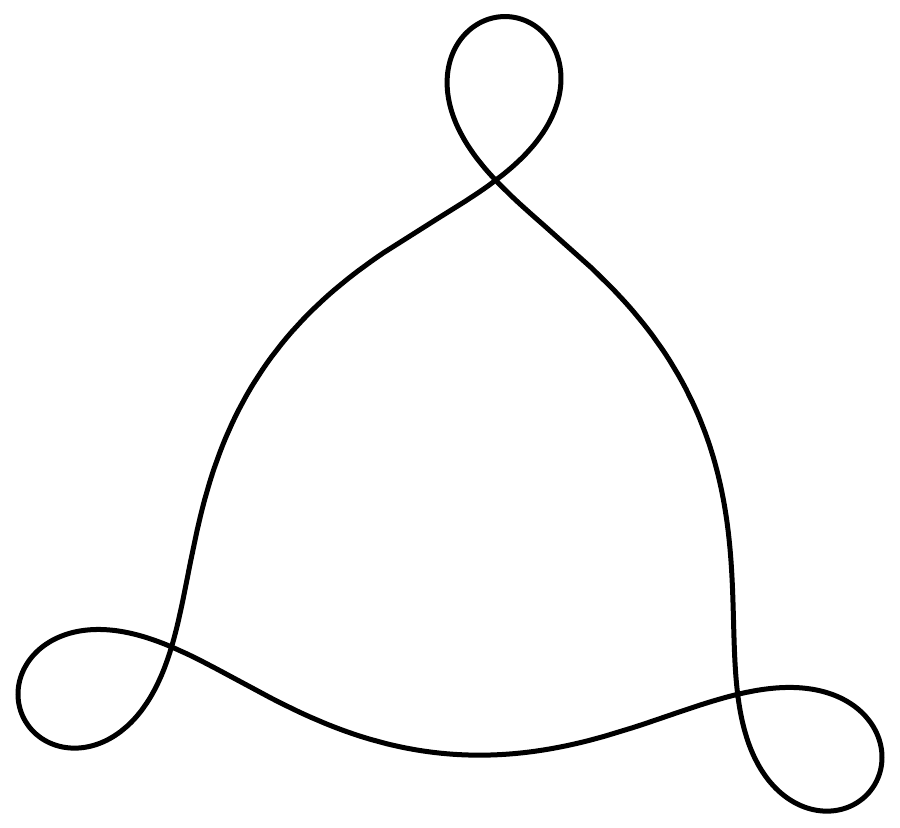}\\
	\includegraphics[width=1.85cm]{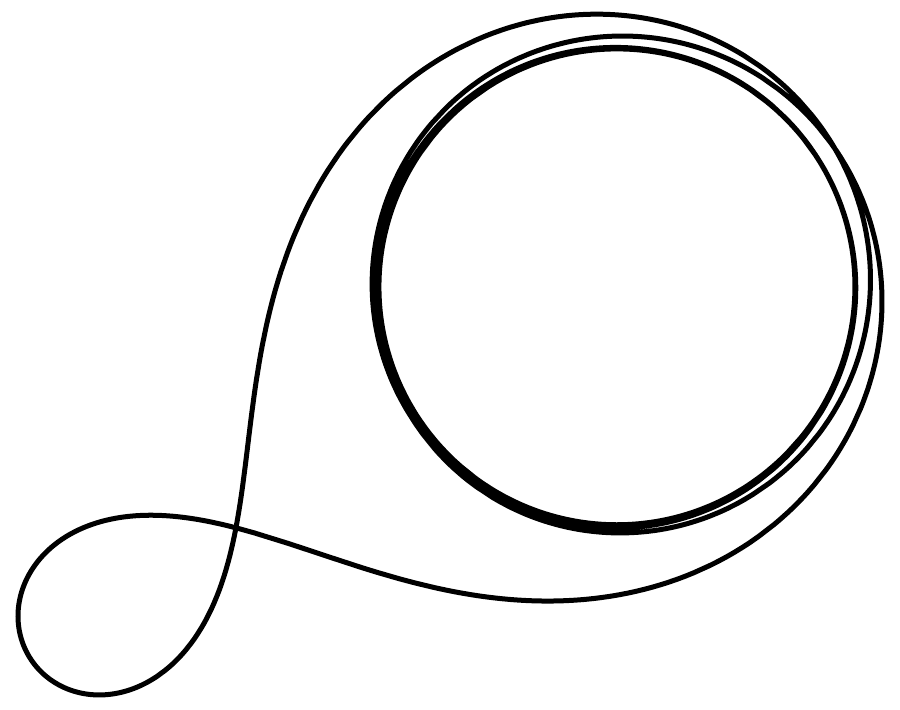}
  	\includegraphics[width=1.85cm]{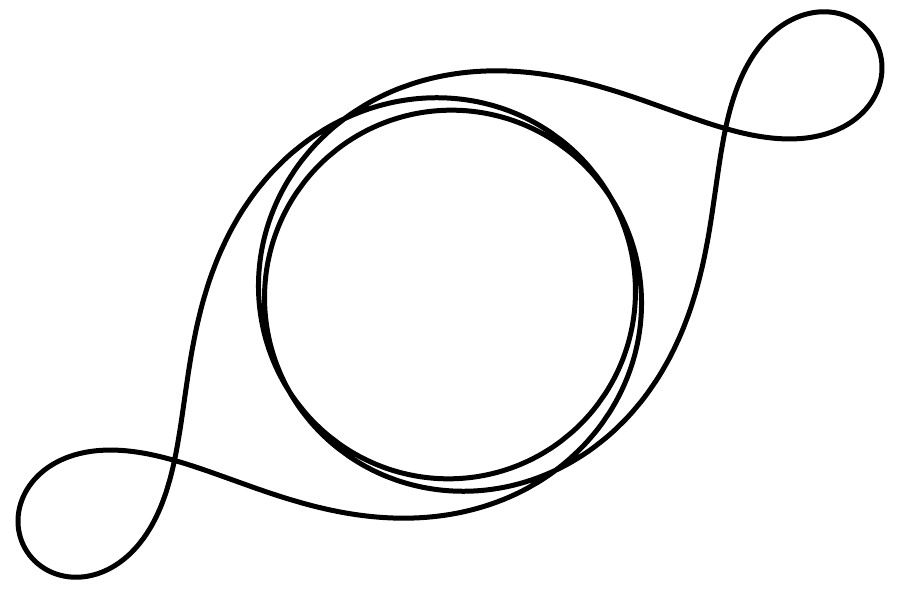}
  	\includegraphics[width=1.85cm]{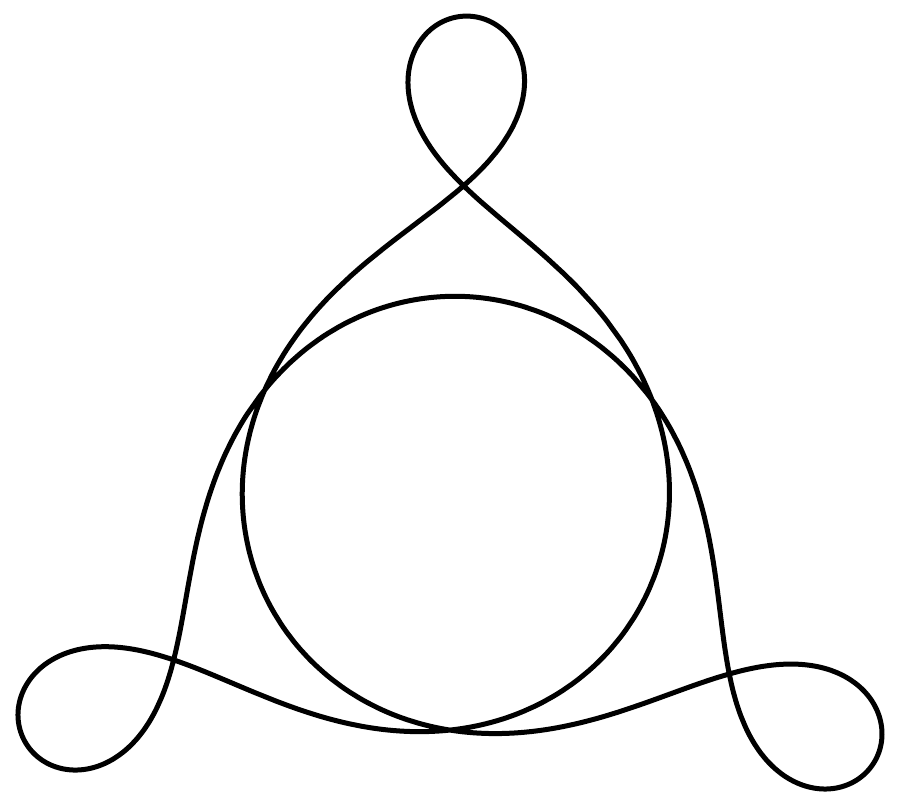}
	\includegraphics[width=1.85cm]{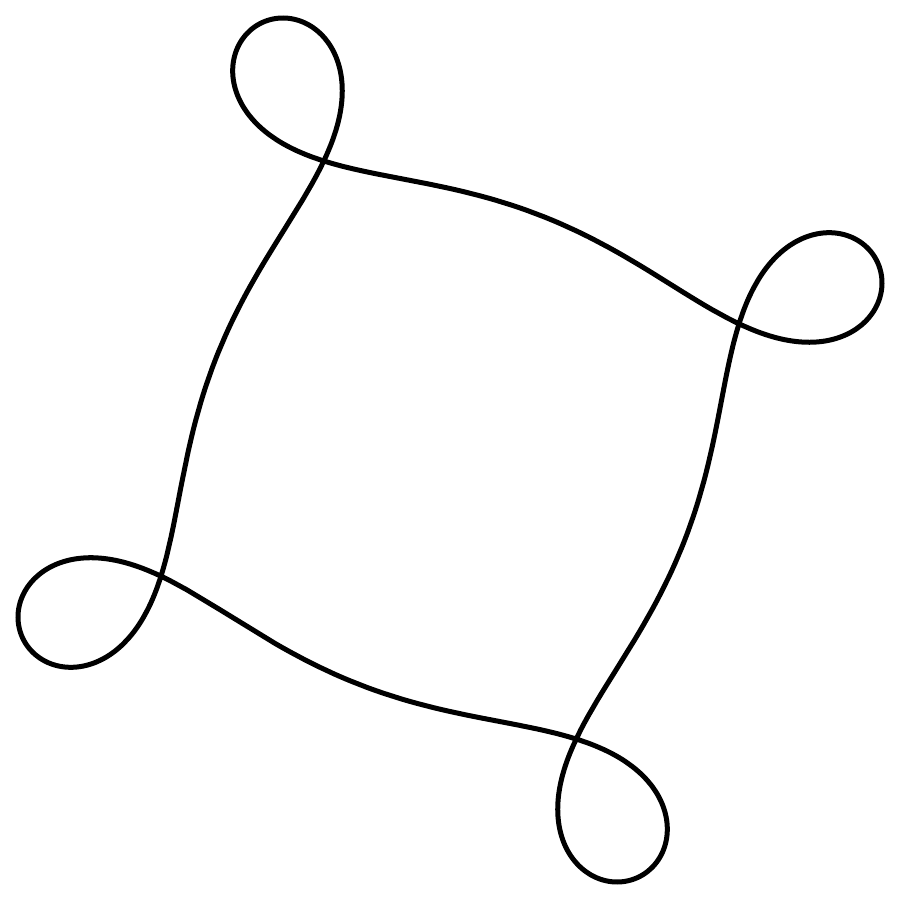}\\
	\includegraphics[width=1.85cm]{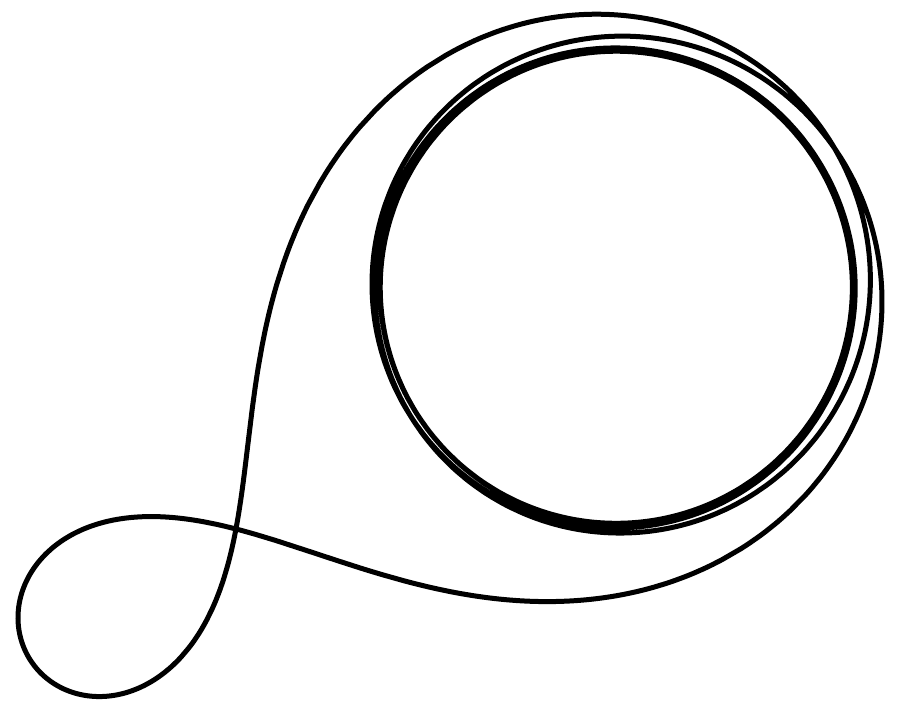}
  	\includegraphics[width=1.85cm]{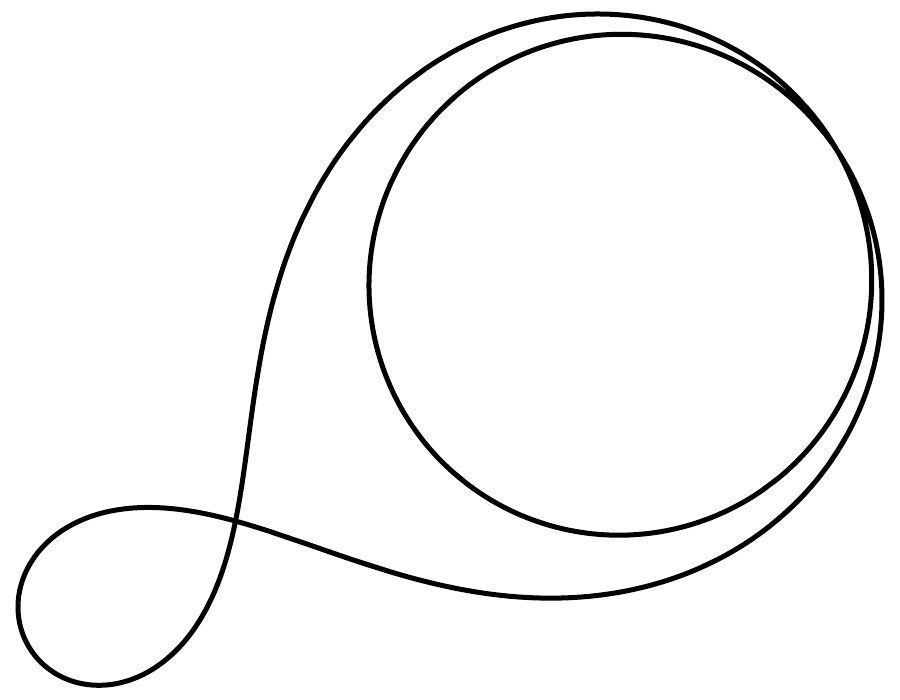}
  	\includegraphics[width=1.85cm]{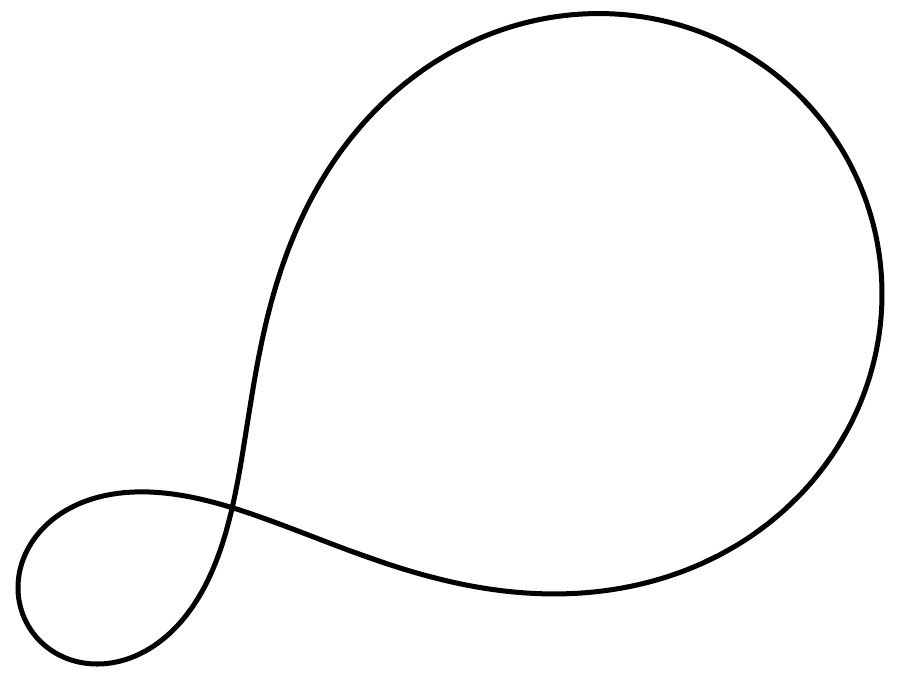}
	\includegraphics[width=1.85cm]{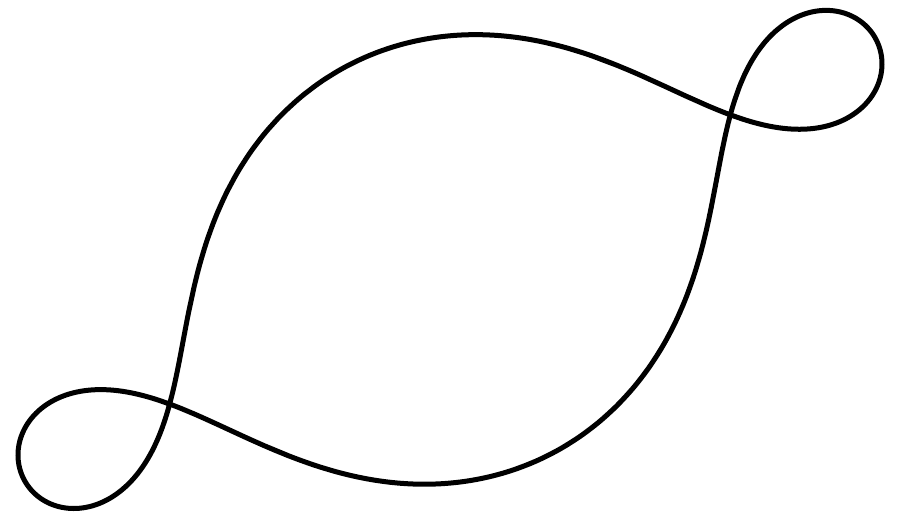}
	\includegraphics[width=1.85cm]{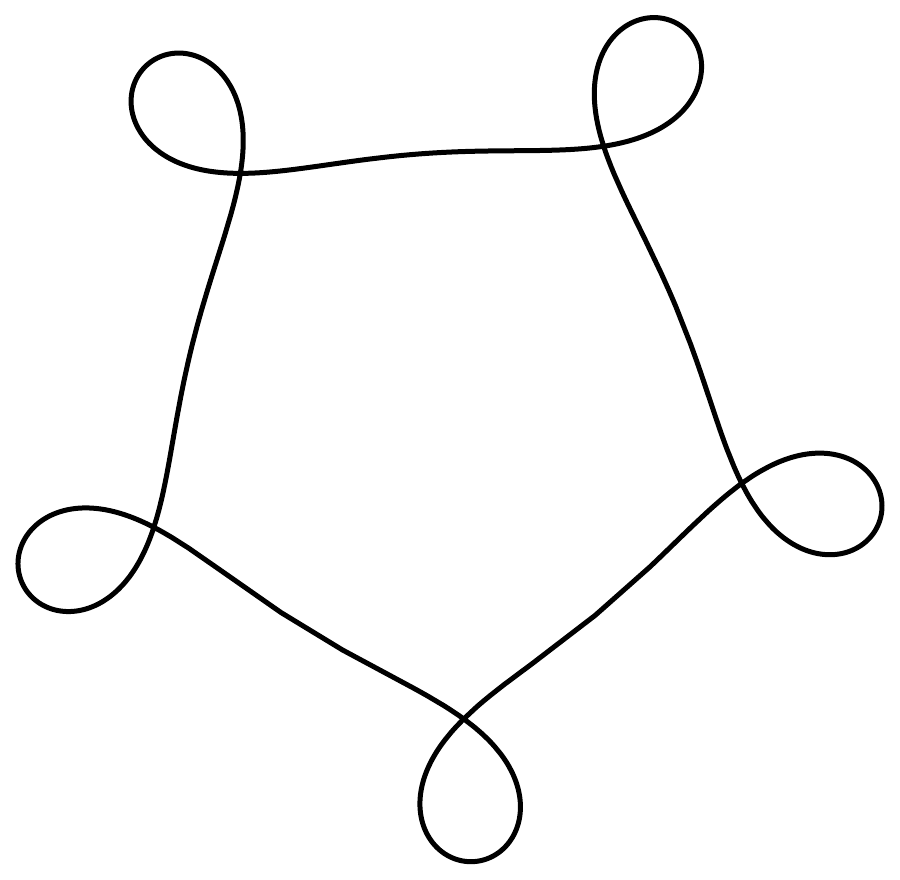}\\
	\includegraphics[width=1.85cm]{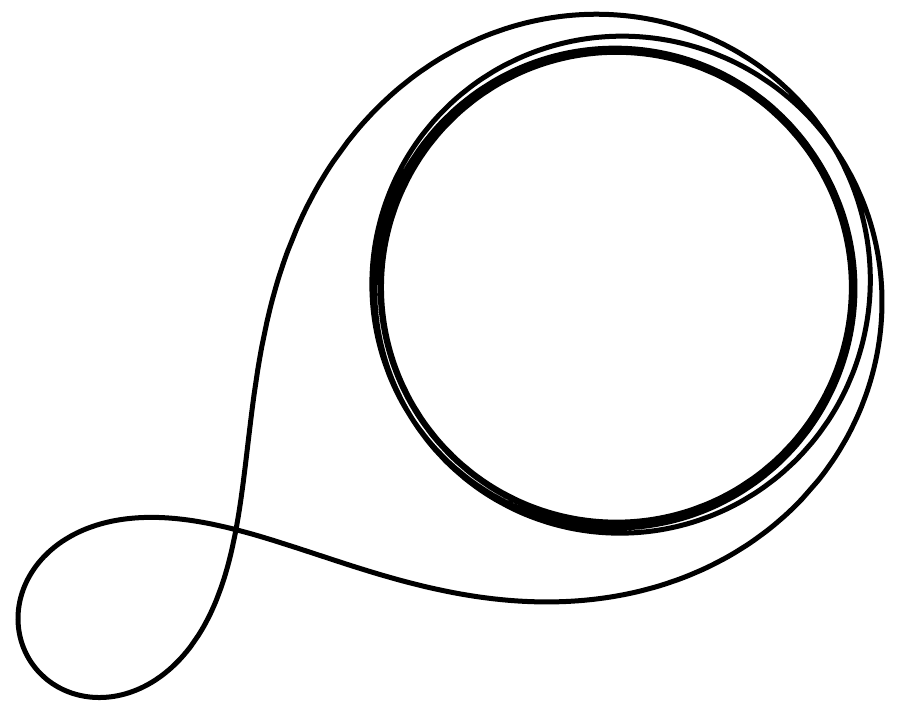}
  	\includegraphics[width=1.85cm]{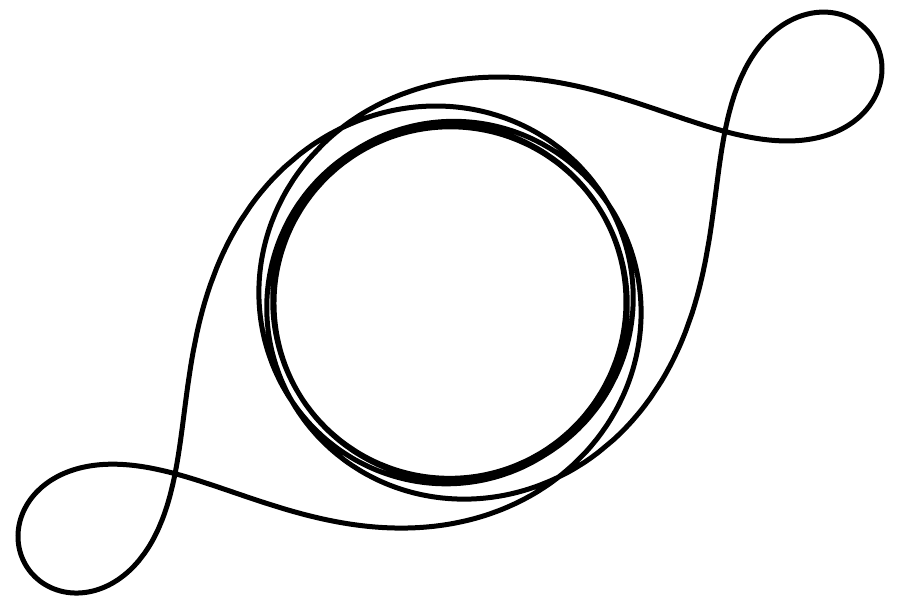}
  	\includegraphics[width=1.85cm]{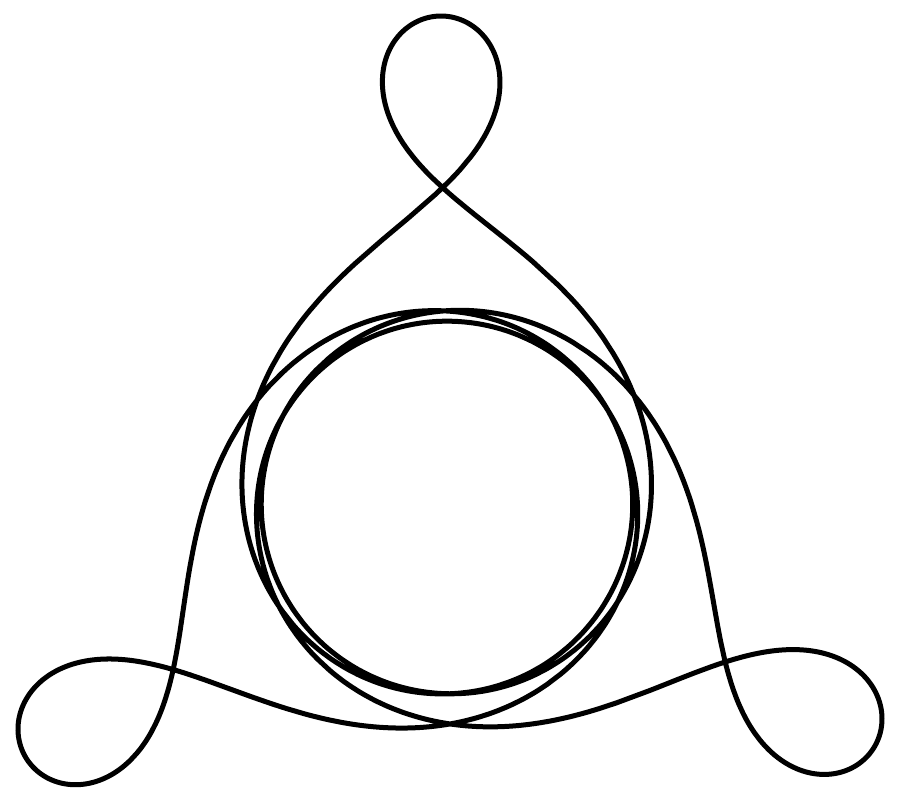}
	\includegraphics[width=1.85cm]{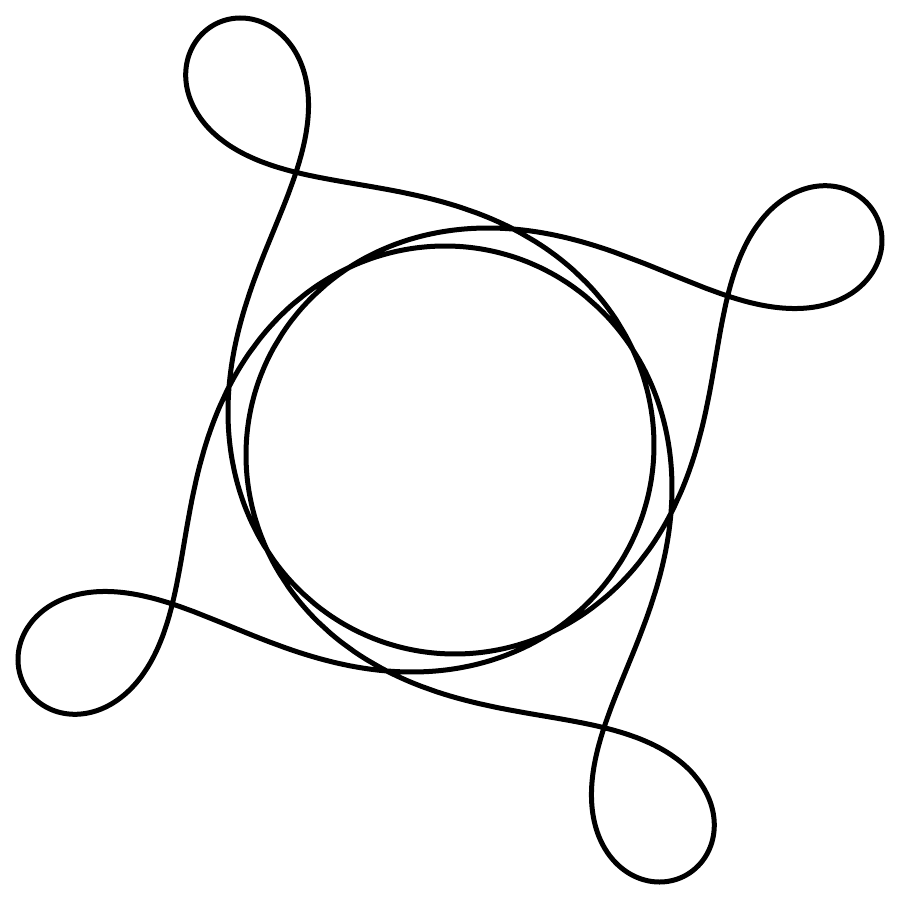}
	\includegraphics[width=1.85cm]{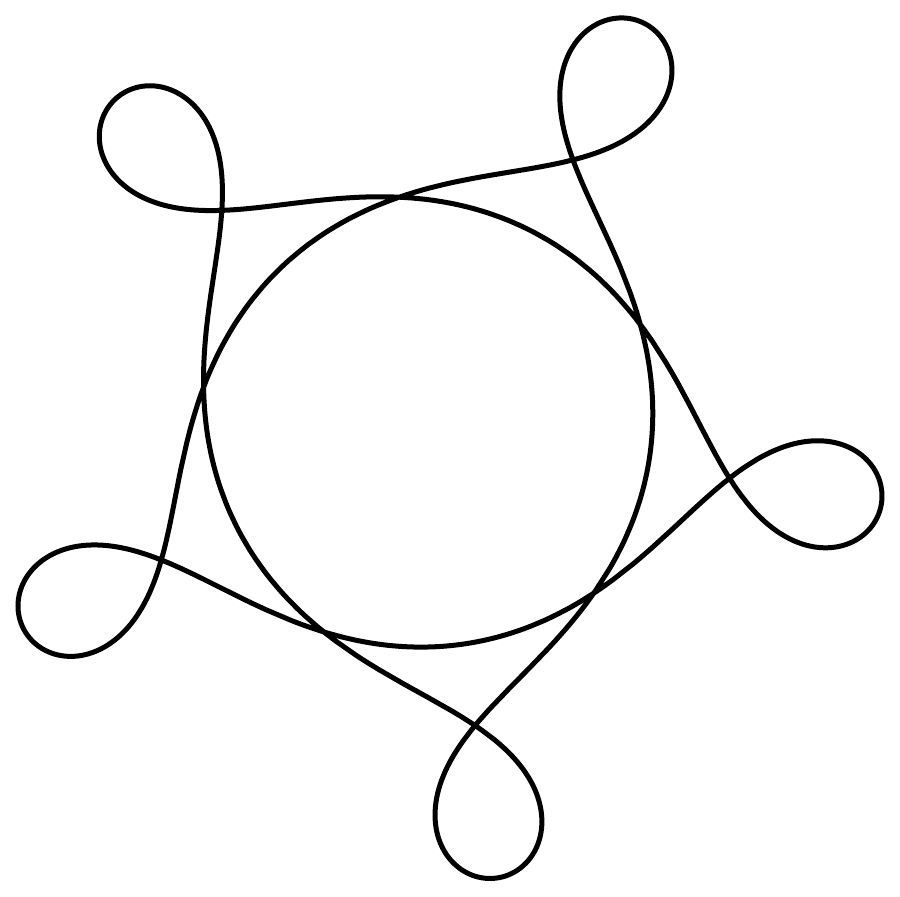}
	\includegraphics[width=1.85cm]{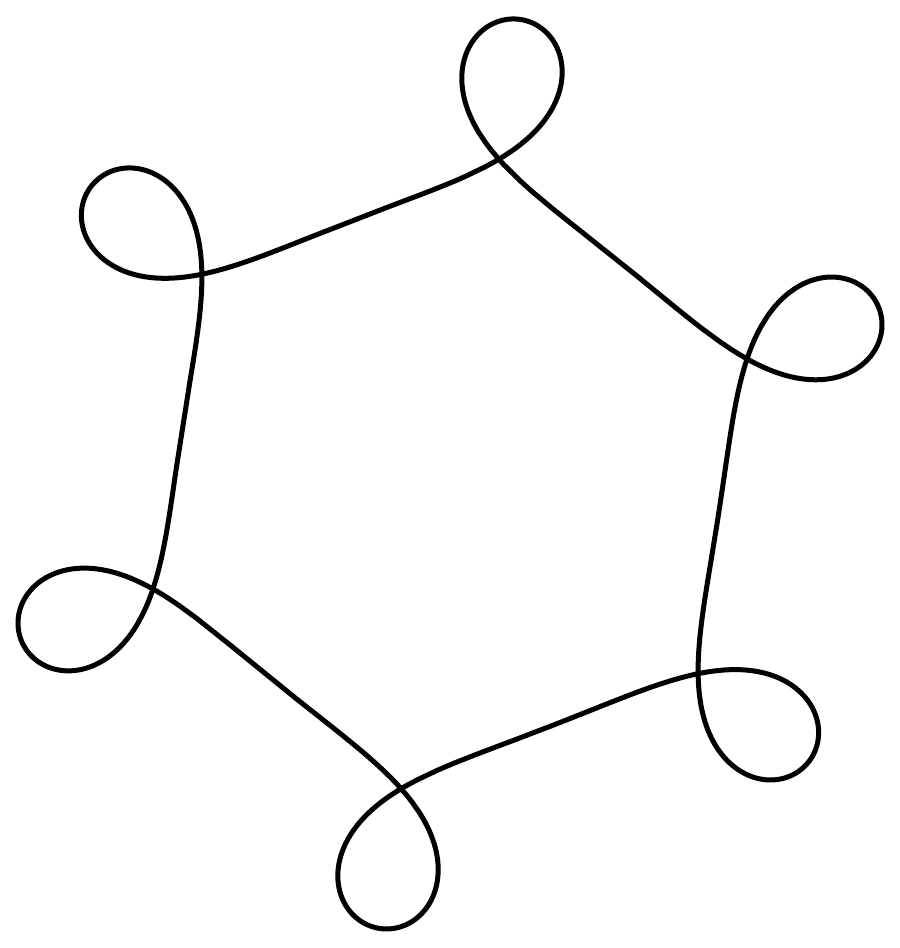}\\
	\includegraphics[width=1.85cm]{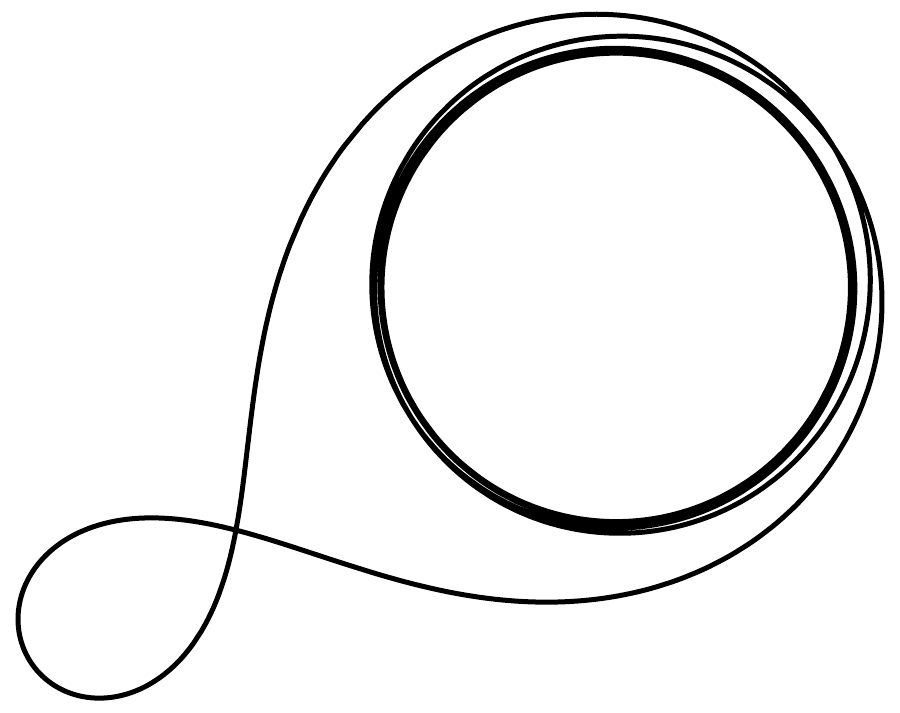}
  	\includegraphics[width=1.85cm]{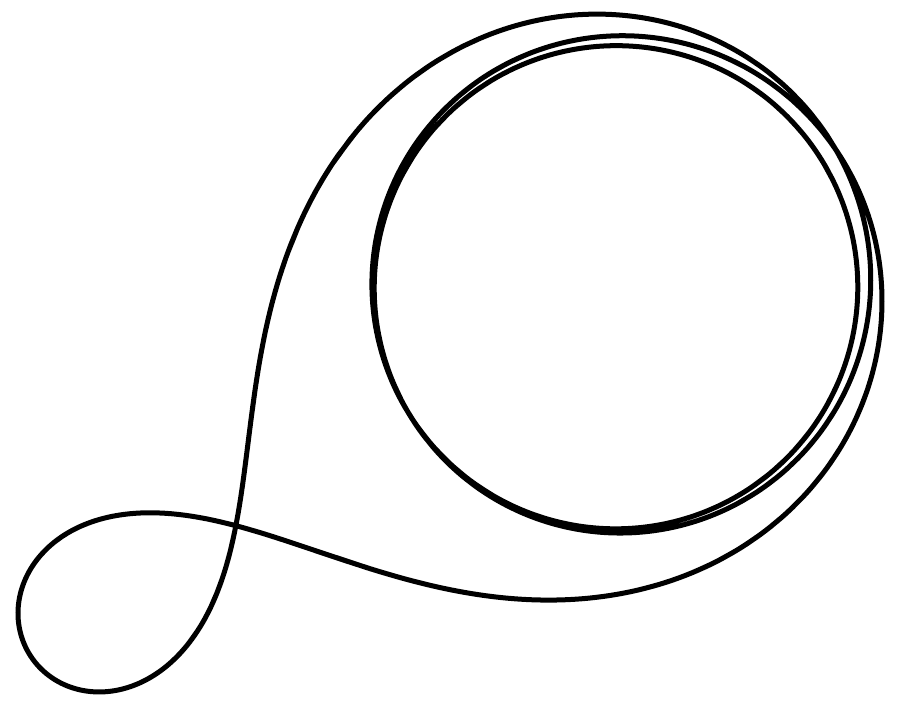}
  	\includegraphics[width=1.85cm]{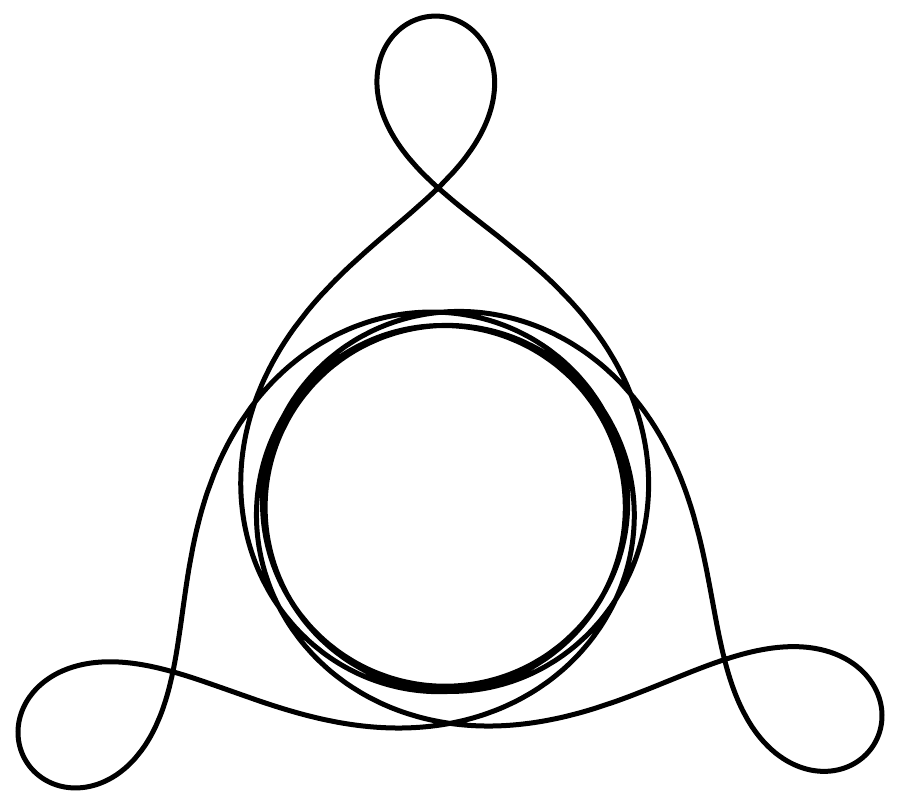}
	\includegraphics[width=1.85cm]{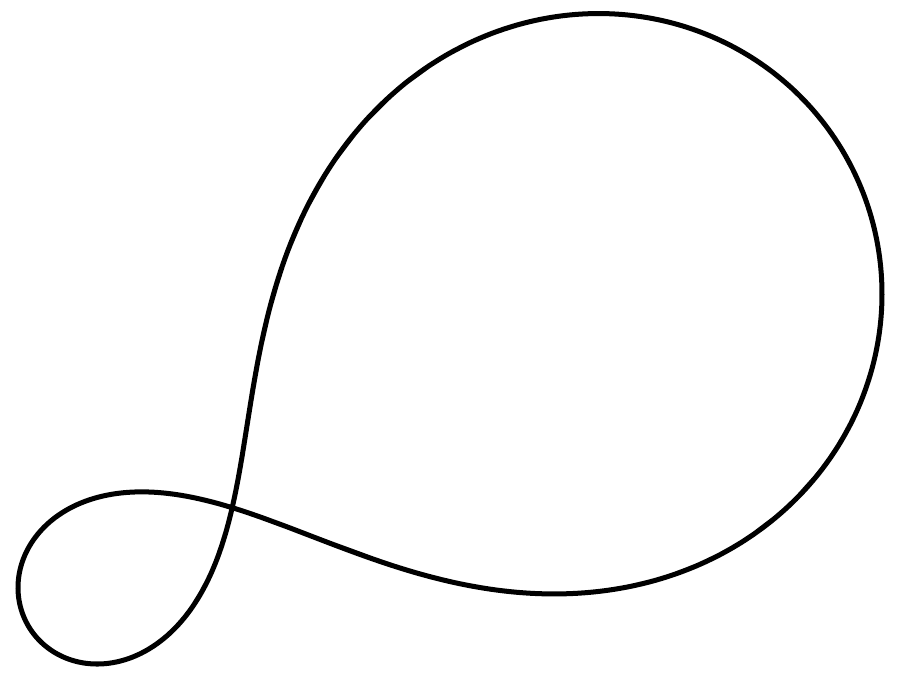}
	\includegraphics[width=1.85cm]{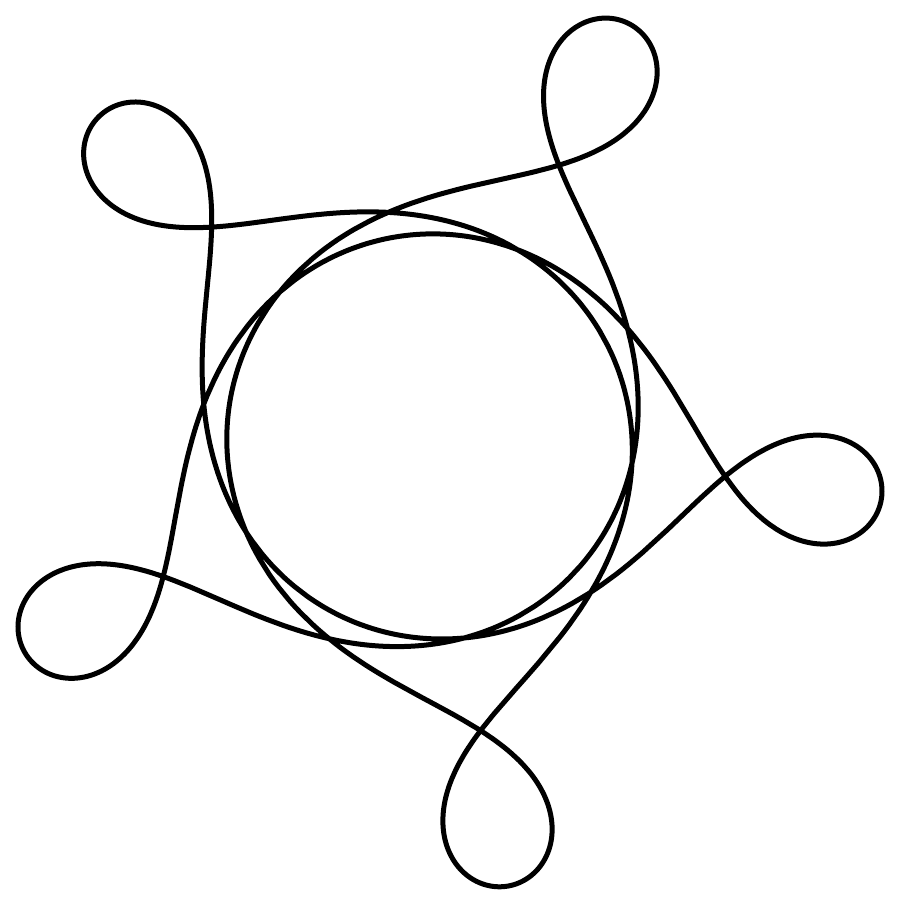}
	\includegraphics[width=1.85cm]{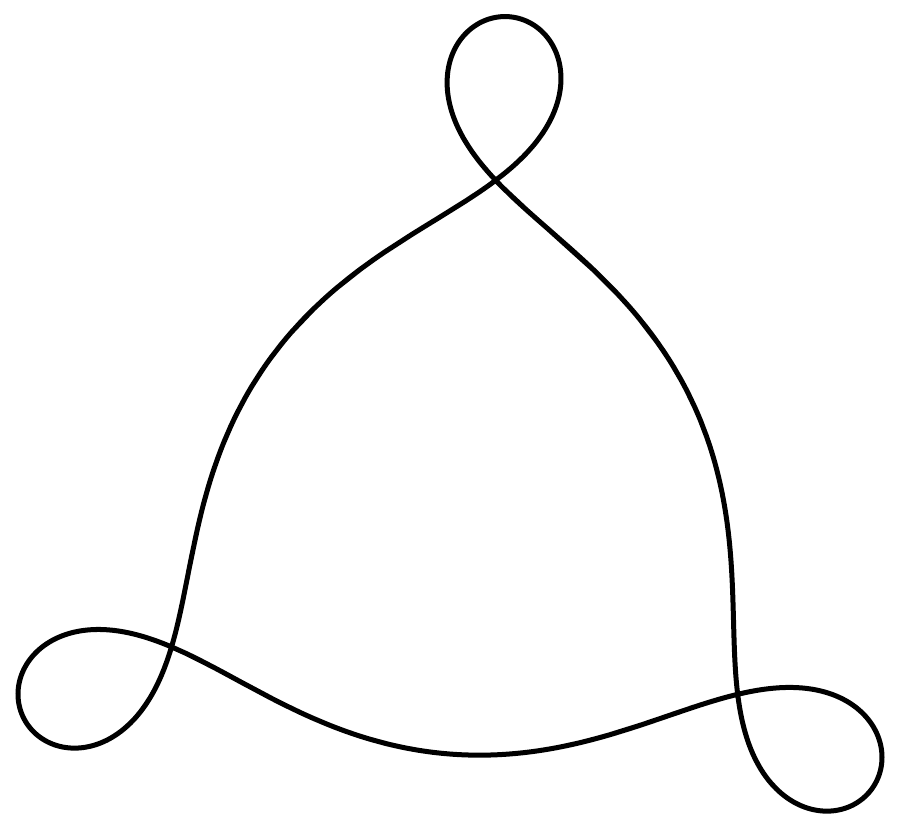}
	\includegraphics[width=1.85cm]{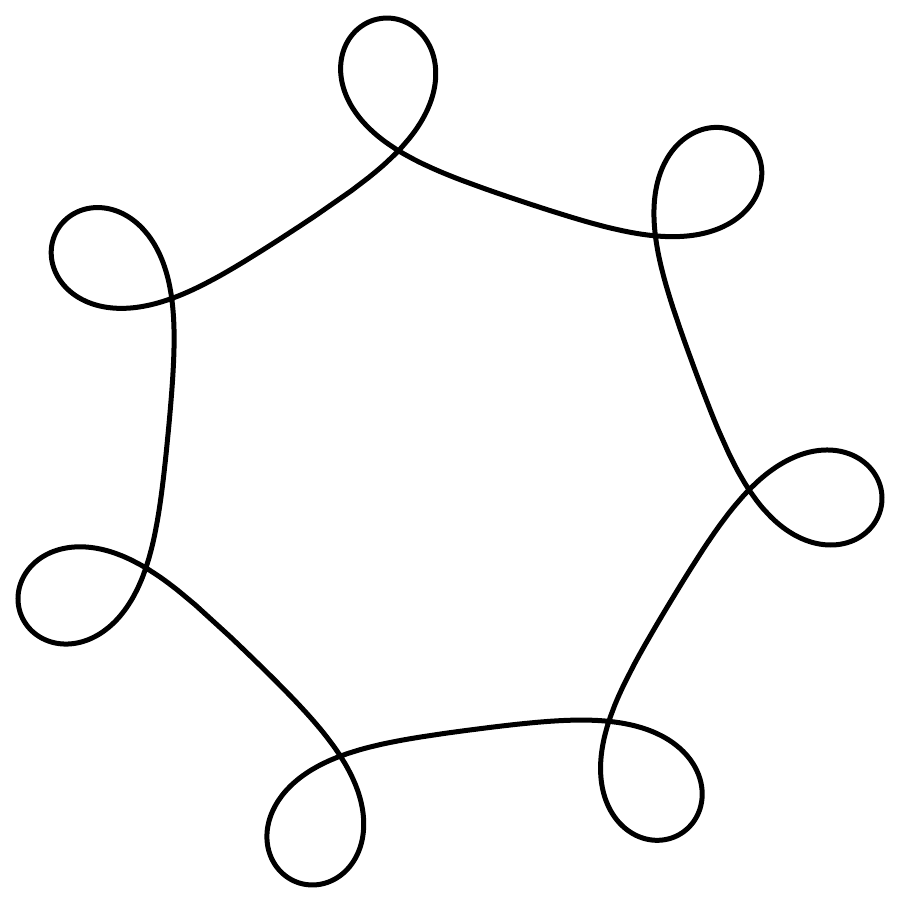}\\
	\includegraphics[width=1.85cm]{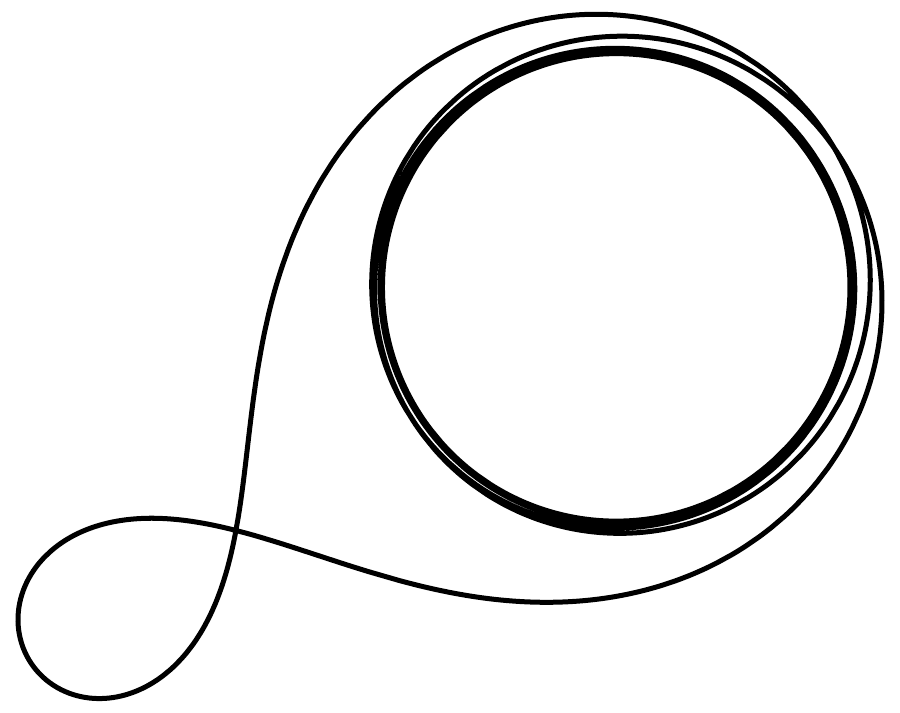}
  	\includegraphics[width=1.85cm]{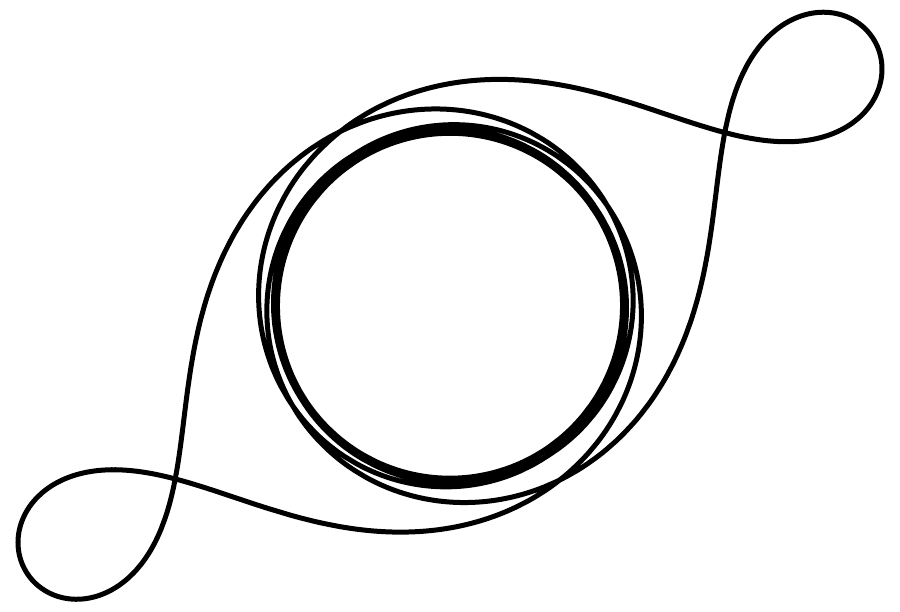}
  	\includegraphics[width=1.85cm]{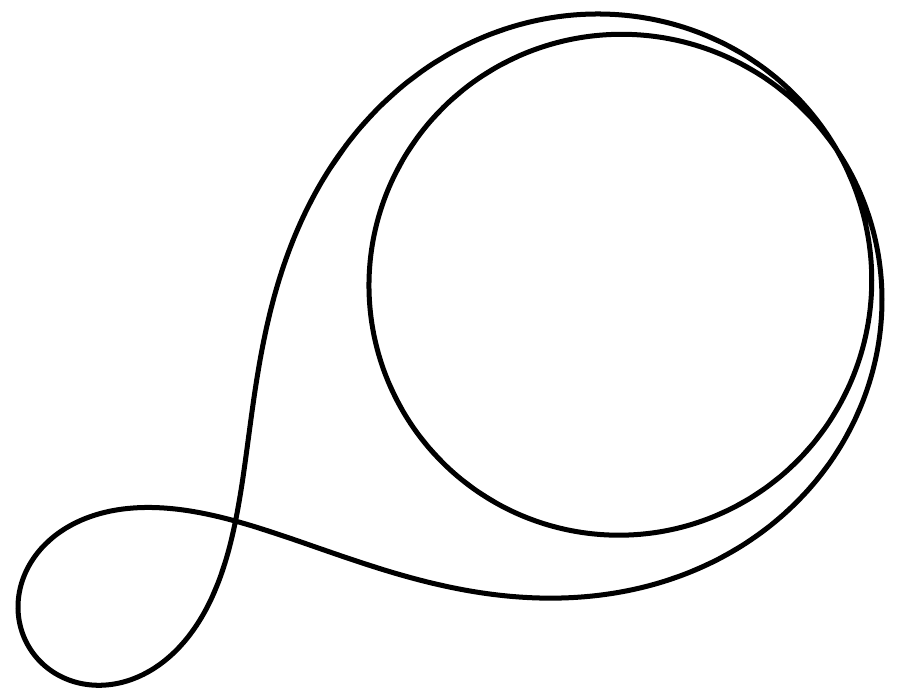}
	\includegraphics[width=1.85cm]{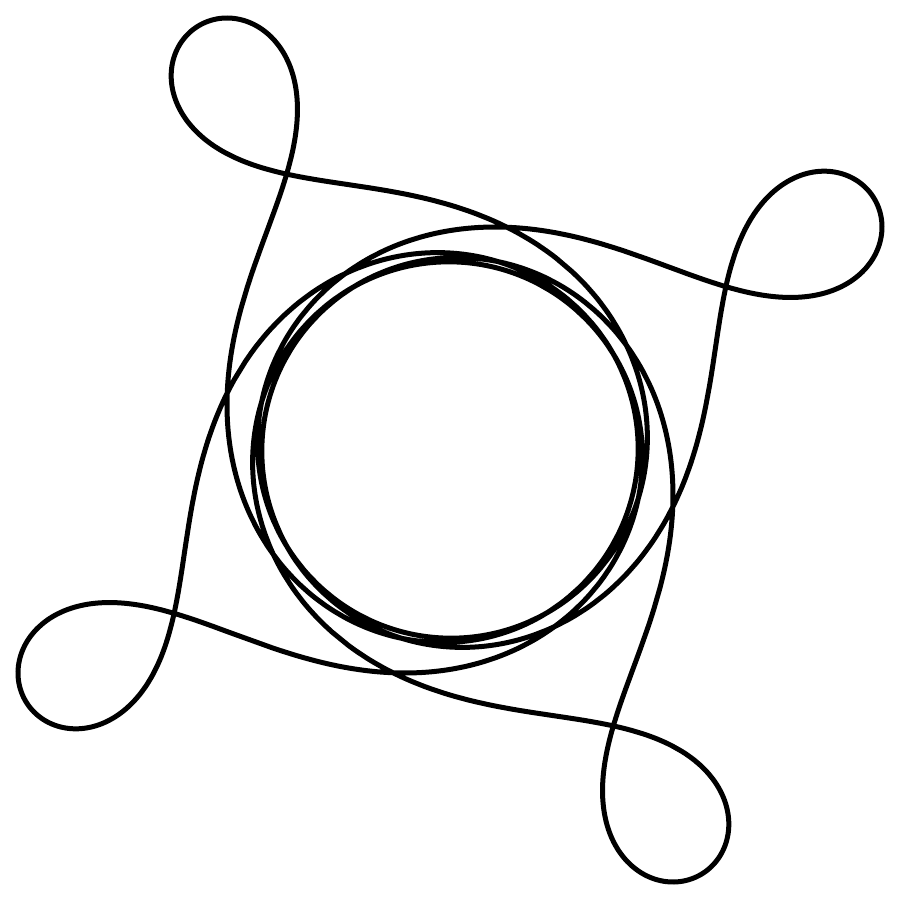}
	\includegraphics[width=1.85cm]{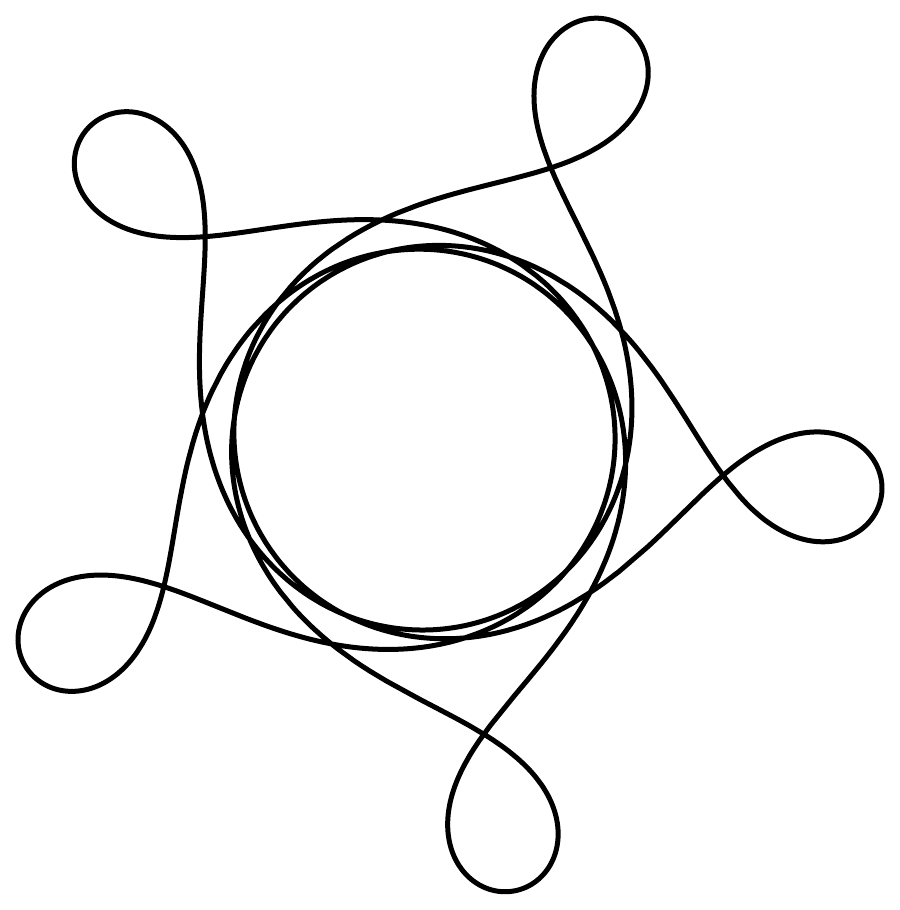}
	\includegraphics[width=1.85cm]{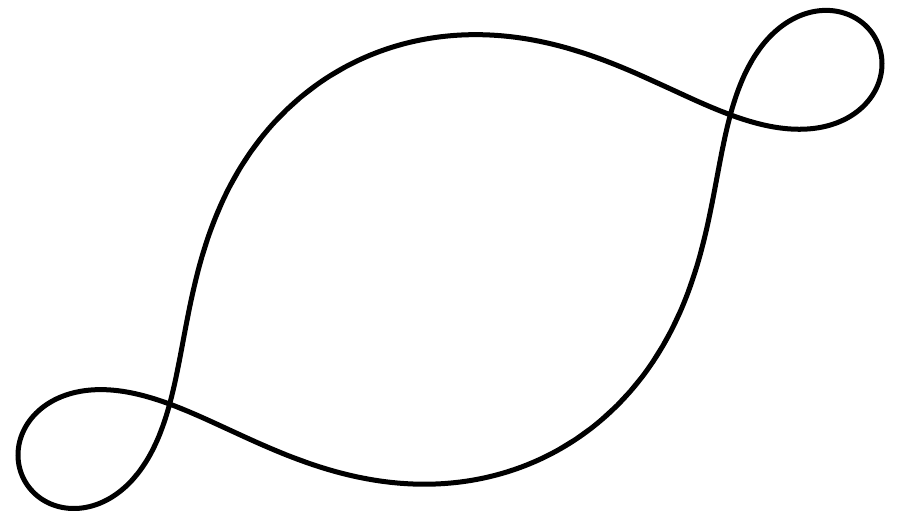}
	\includegraphics[width=1.85cm]{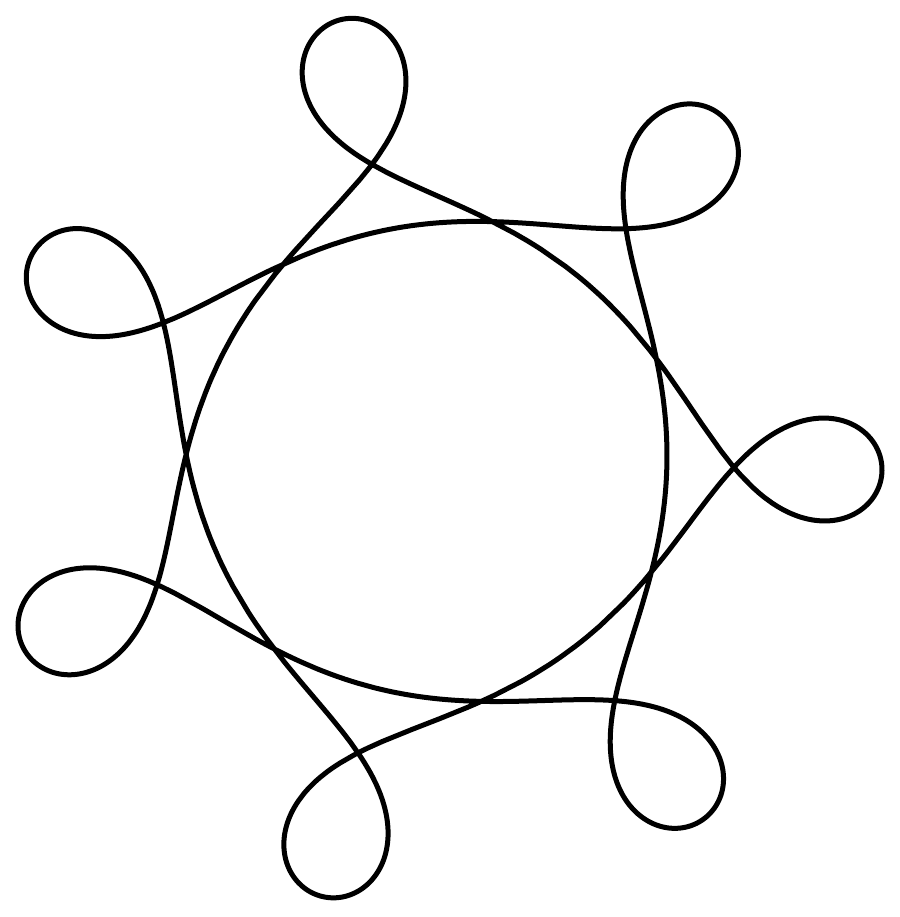}
	\includegraphics[width=1.85cm]{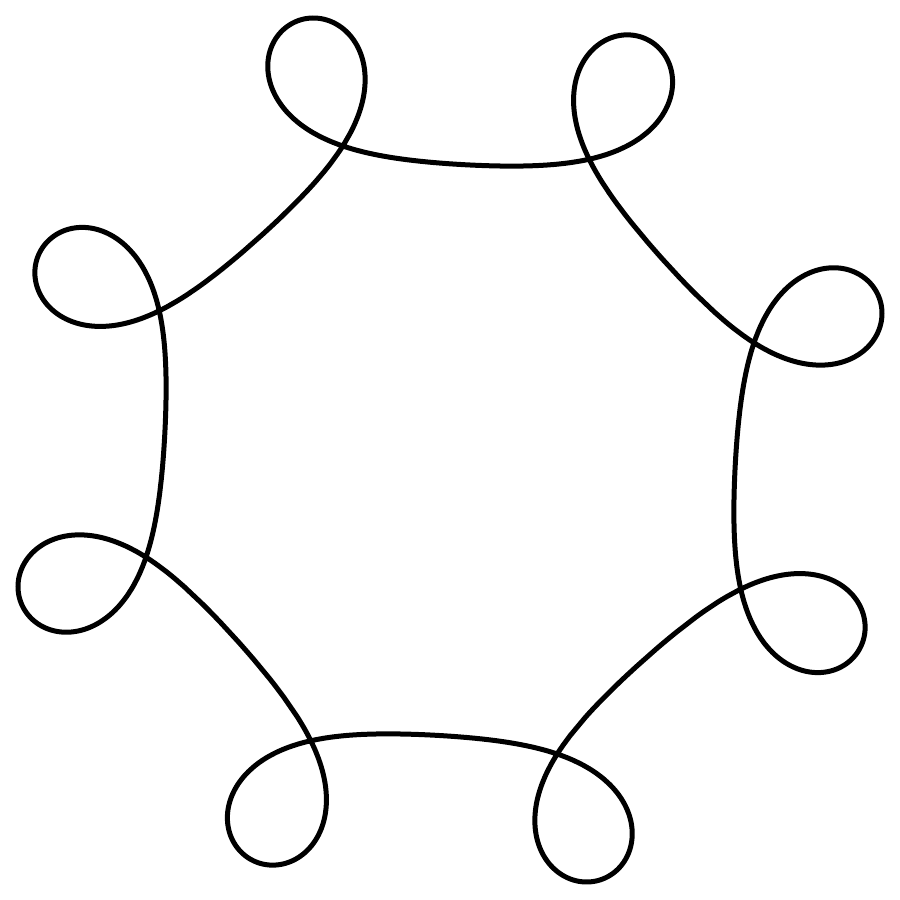}
\caption{ \label{fig:single-circletons} Single-circletons: In the first row the only deformation of the twice-wrapped circle. In the second row the two deformations of the thrice-wrapped circle, and so on.}
\end{figure}
A $(k,\,\omega)$-circleton is the curve obtained by dressing a circle of radius 1 with wrapping number $\omega$ by the simple factor of Lemma~\ref{th:circleton}. By the next result, the bending energy does not change.
\begin{figure}[t] 
  	\includegraphics[width=4cm]{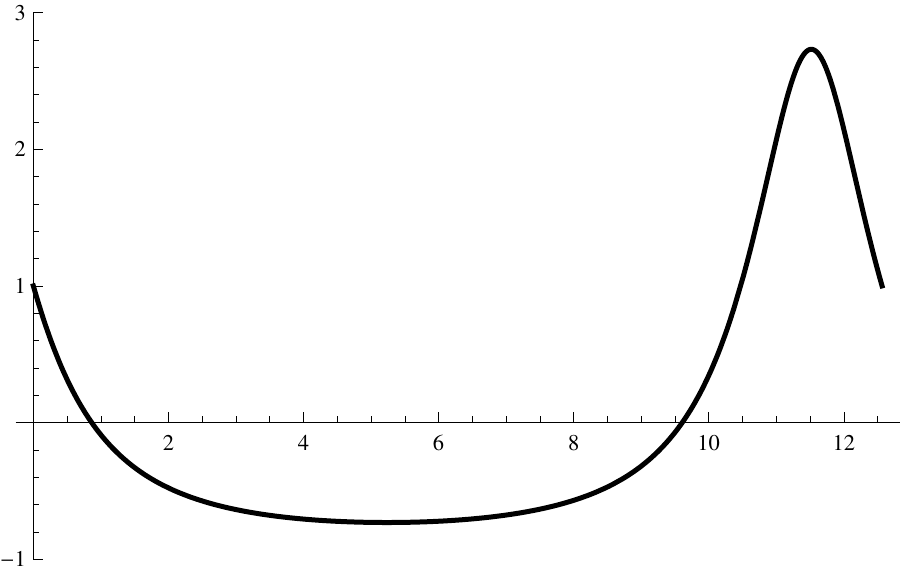}\\
  	\includegraphics[width=4cm]{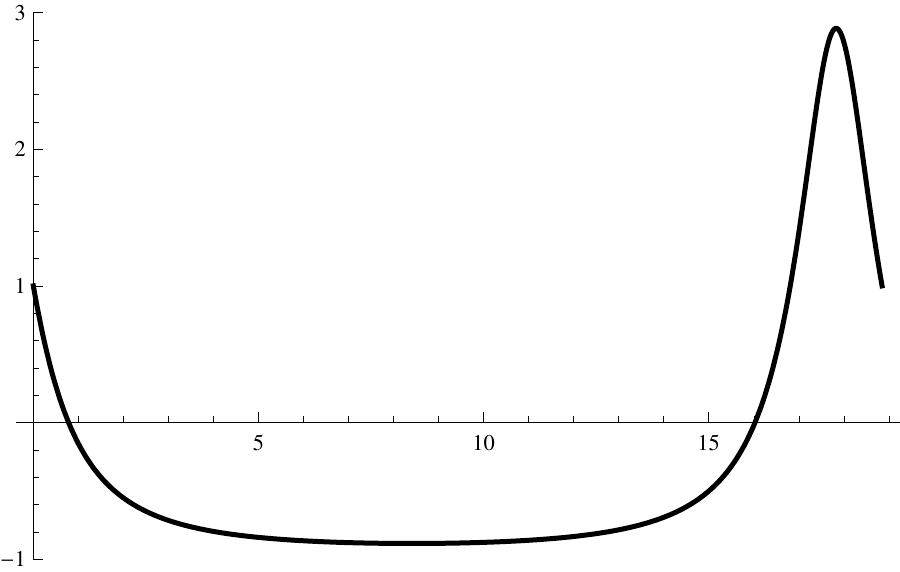}
  	\includegraphics[width=4cm]{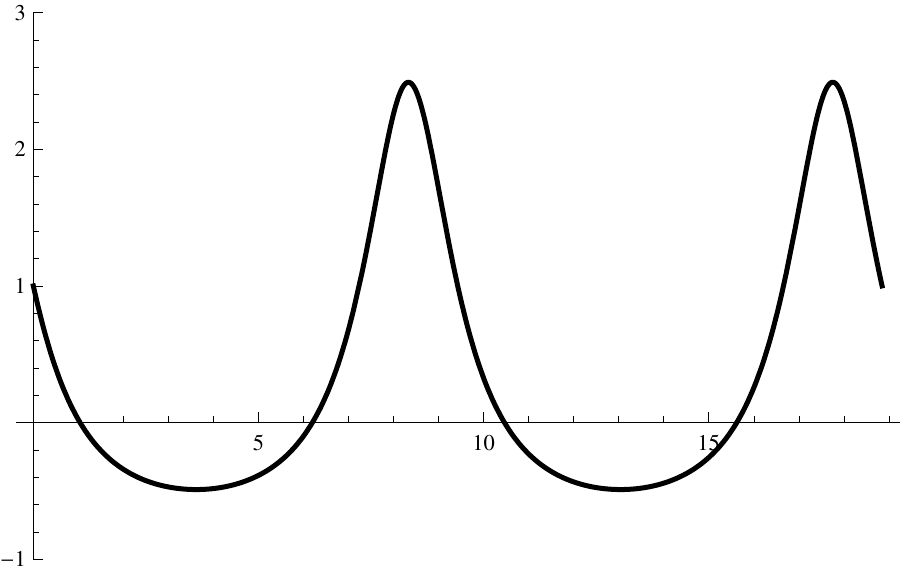}\\
	\includegraphics[width=4cm]{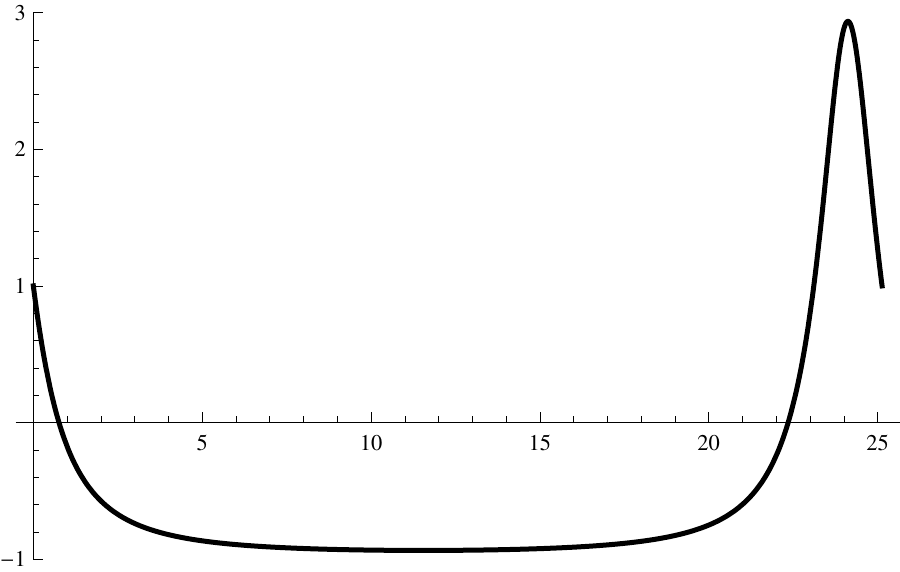}
  	\includegraphics[width=4cm]{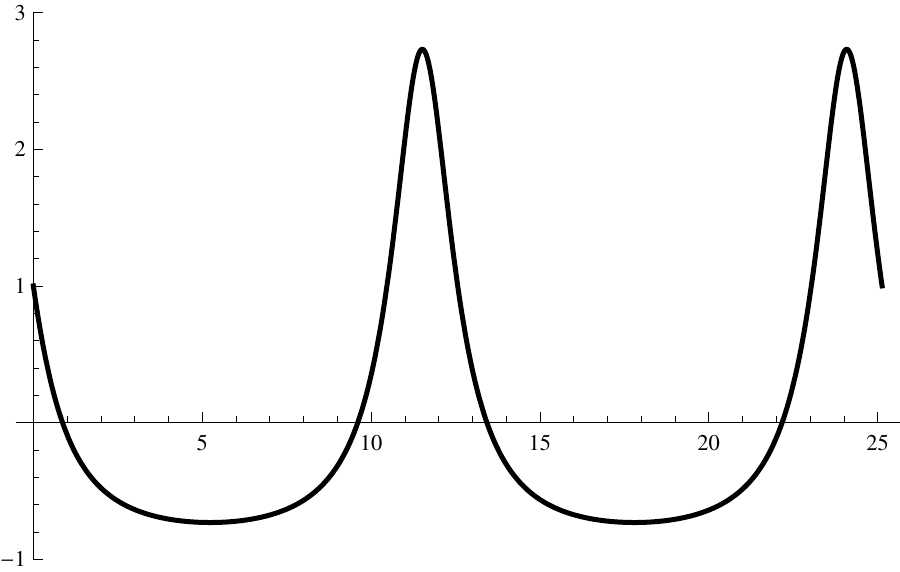}
  	\includegraphics[width=4cm]{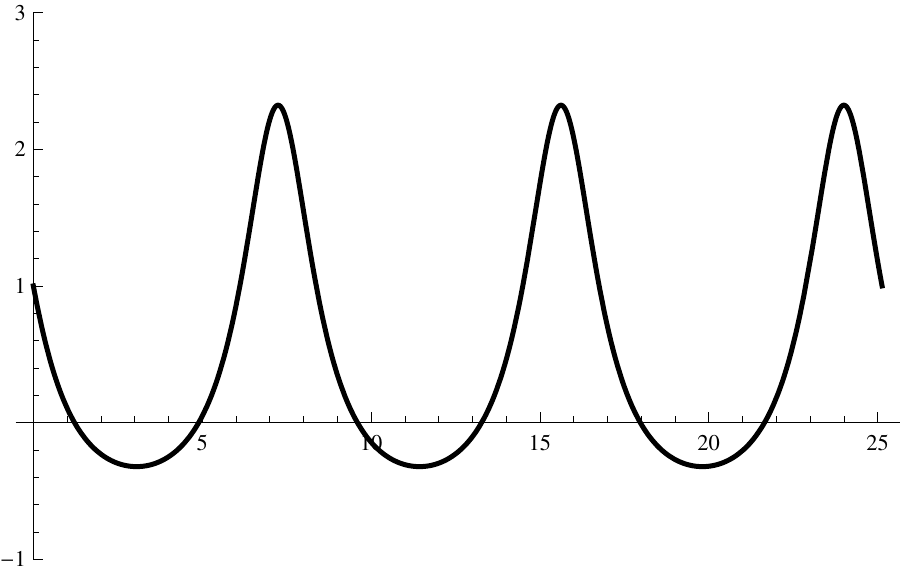}\\
	\includegraphics[width=4cm]{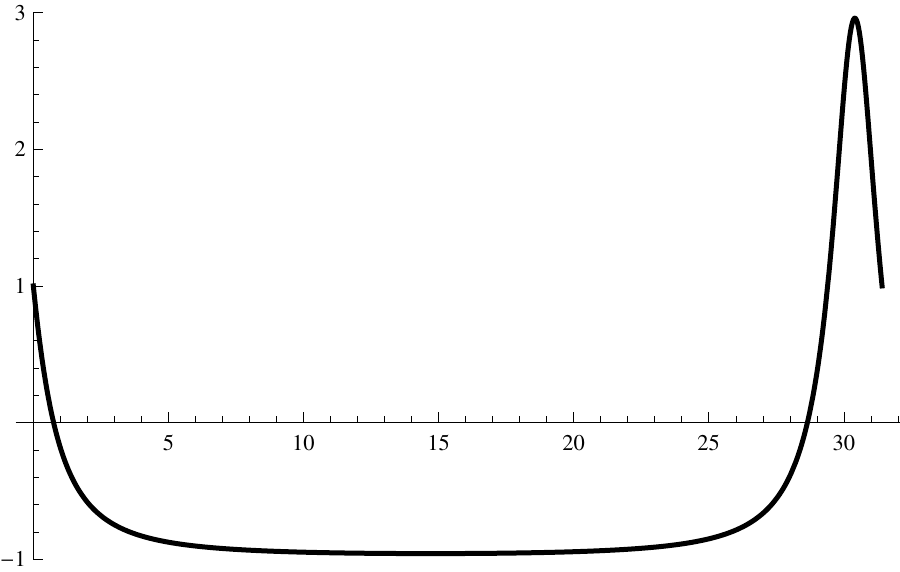}
  	\includegraphics[width=4cm]{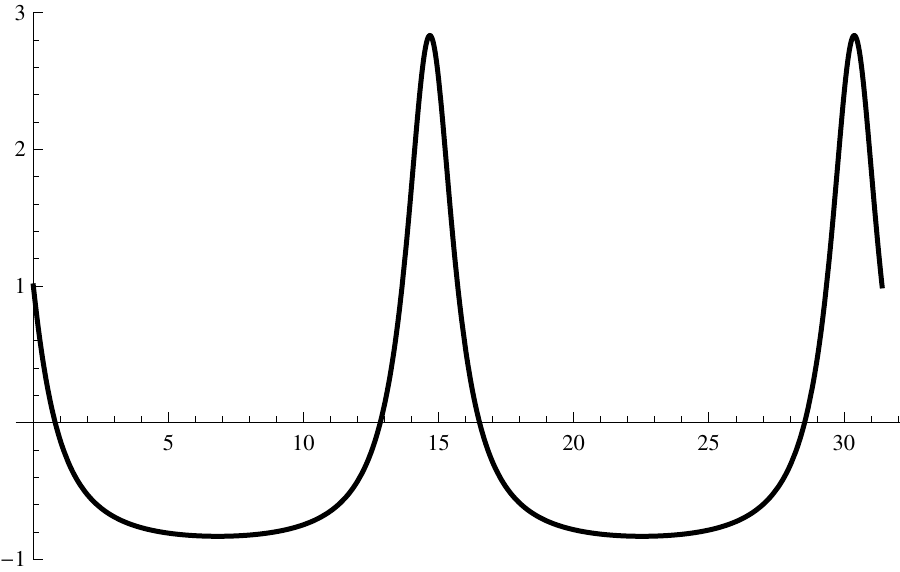}
  	\includegraphics[width=4cm]{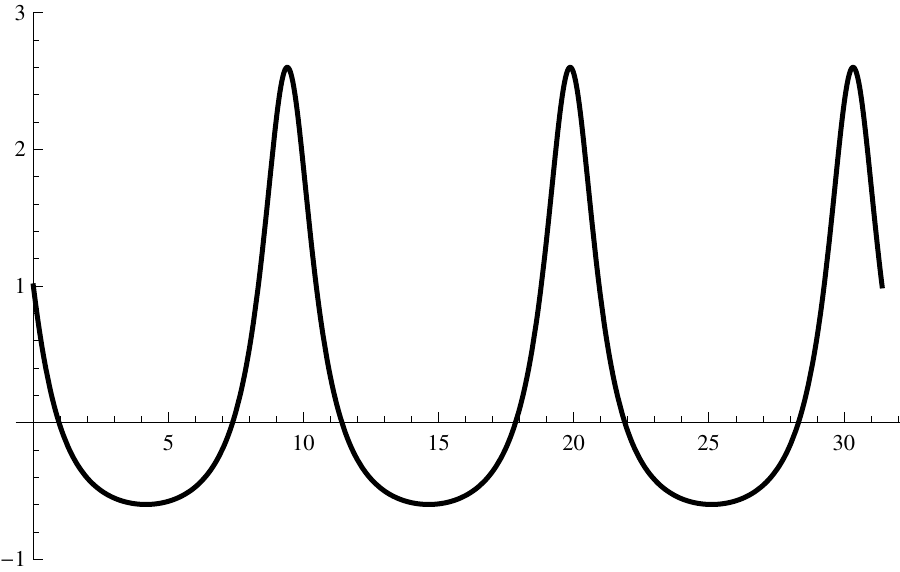}
	\includegraphics[width=4cm]{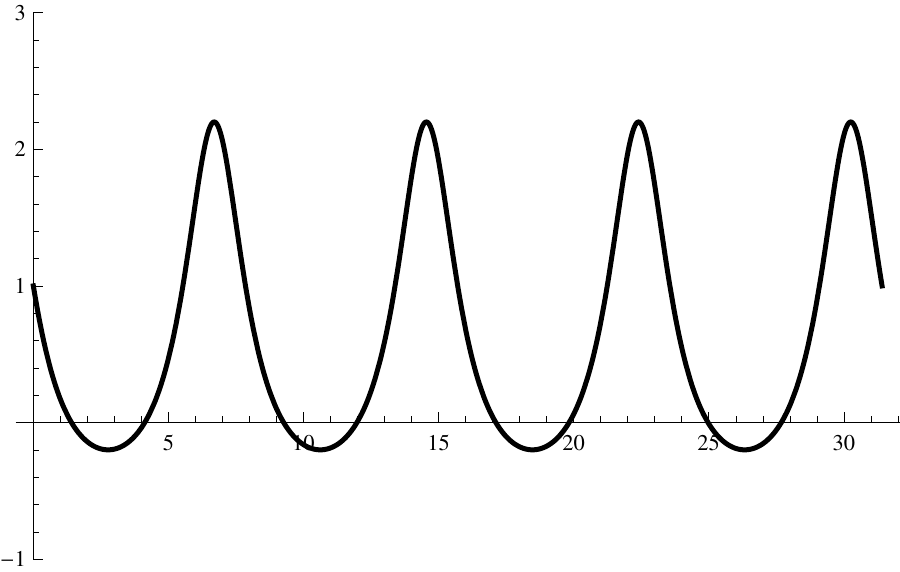}
\caption{ \label{fig:curvatures}
Plots of curvature functions of single-circletons over one extrinsic period $[0,\,2\pi \omega ]$: On top the curvature function of the $(1,2)$-circleton. In the second row the curvature functions of the $(1,\,3)$-circleton and the $(2,\,3)$-circleton, and so on.}
\end{figure}
\begin{proposition}
The curvature function of the $(k,\,\omega)$-circleton has $k$ peaks on the interval $[0,\,2\omega \pi]$. The bending energy of a  $(k,\,\omega)$-circleton is
\[
	\int_0^{2\omega\pi} \kappa^2(t) \,dt = 2\omega \pi\,.
\]
The curvature function $\kappa$ of a single-circleton satisfies $-1 < \kappa (t) <3$.
\end{proposition}
\begin{proof}
The claims follows form the explicit formula for the curvature of a $(k,\,\omega)$-circleton given by 
\[
	\kappa (t) = \tfrac{(k^2 - \omega^2 ) \left(2 k \omega  \sqrt{1-\frac{k^2}{\omega ^2}} \left(2 \sin \left(\frac{k t}{\omega }\right)-\sin \left(\frac{2 k t}{\omega }\right)\right)+\left(\omega
   ^2-2 k^2\right) \cos \left(\frac{2 k t}{\omega }\right)+4 (k^2 - \omega^2 ) \cos \left(\frac{k t}{\omega }\right)\right)+5 k^2 \omega ^2-3 \omega ^4}{\omega  \left(4 k
   \sqrt{1-\frac{k^2}{\omega ^2}} \sin \left(\frac{k t}{\omega }\right) \left((k^2 -\omega^2 ) \cos \left(\frac{k t}{\omega }\right)+\omega ^2\right)-k^2 \omega +3 \omega
   ^3\right)+\left(2 k^4-3 k^2 \omega ^2+\omega ^4\right) \cos \left(\frac{2 k t}{\omega }\right)+4 \omega ^2 (k^2 -\omega^2 )  \cos \left(\frac{k t}{\omega }\right)}\,,
\]
which is obtained from $(h_{L,\alpha} \# F_\lambda)^{-1} \tfrac{d}{dt} h_{L,\alpha} \# F_\lambda $.
\end{proof}
Each peak of the curvature function corresponds to a localized region of high curvature, so a little loop on the curve, see figures ~\ref{fig:single-circletons} and \ref{fig:curvatures}. 
%
%
\subsection{Multi-circletons}
The classification of multi-circletons is now immediate. The monodromy of a circle with wrapping number $\omega$ has $\omega - 1$ many resonance points. Each resonance point may be used only once as the singularity of a simple factor. The order in which this is done is irrelevant, since simple factors are diagonal. so commute. In conclusion we obtain the following 
\begin{proposition}
For an $\omega$-wrapped circle, there are up to isospectral deformations precisely $\bigl( \begin{smallmatrix} k \\ \omega -1 \end{smallmatrix} \bigr)$ many $k$-multi-circletons.
\end{proposition}
Figures~\ref{fig:double-circletons1} and \ref{fig:double-circletons2} show some double-circletons, which are circles dressed by a product of two simple factors. We use the notation $(k_1,\,k_2;\omega )$ to specify the two resonance points obtained from $k_1$ and $k_2$, and where $\omega$ is the wrapping number of the initial circle. 
\begin{figure}
  	\includegraphics[width=1.85cm]{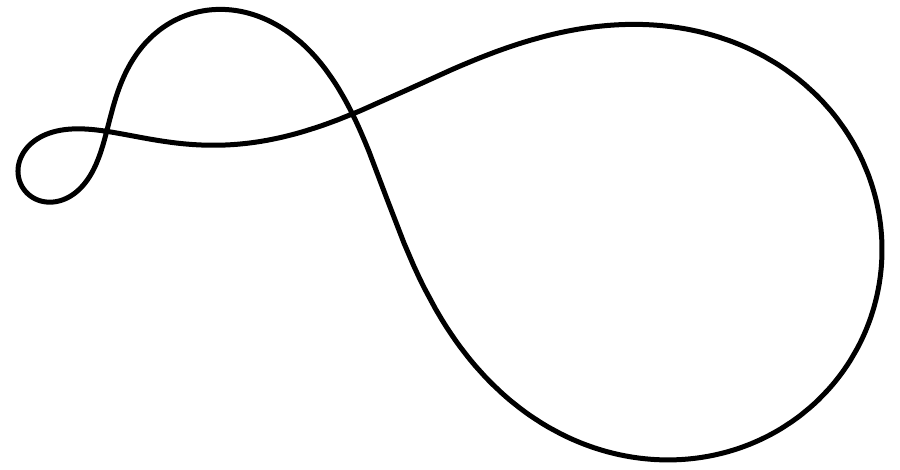}\\
	\includegraphics[width=1.85cm]{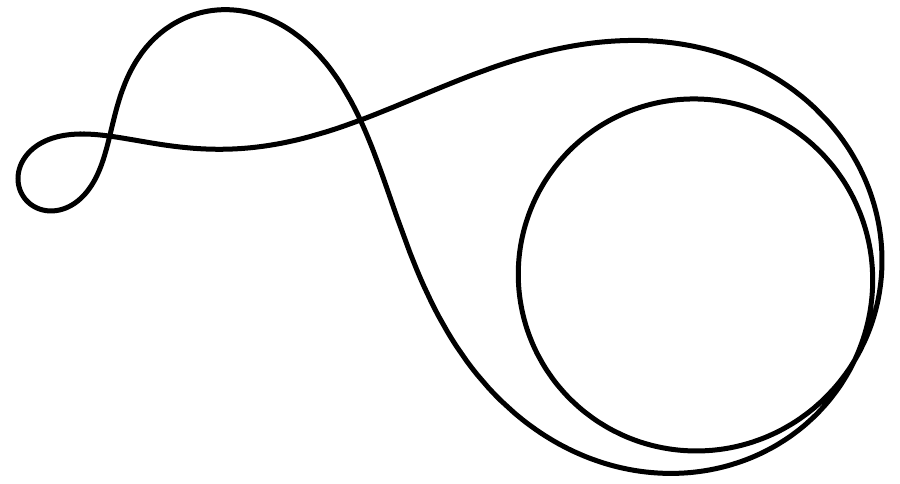}
	\includegraphics[width=1.85cm]{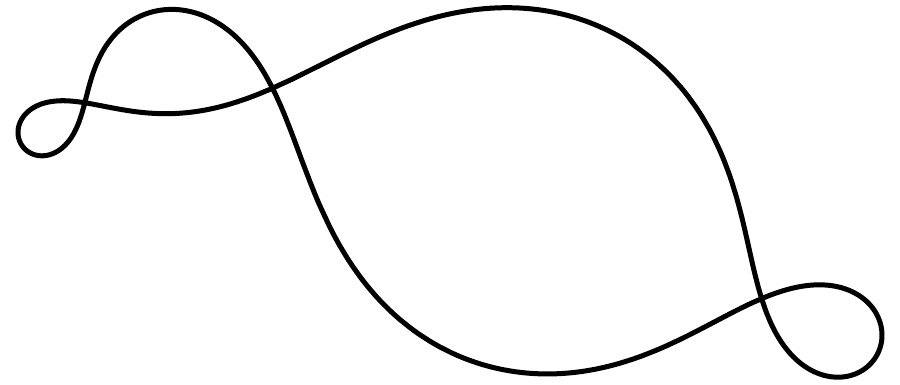}
	\includegraphics[width=1.85cm]{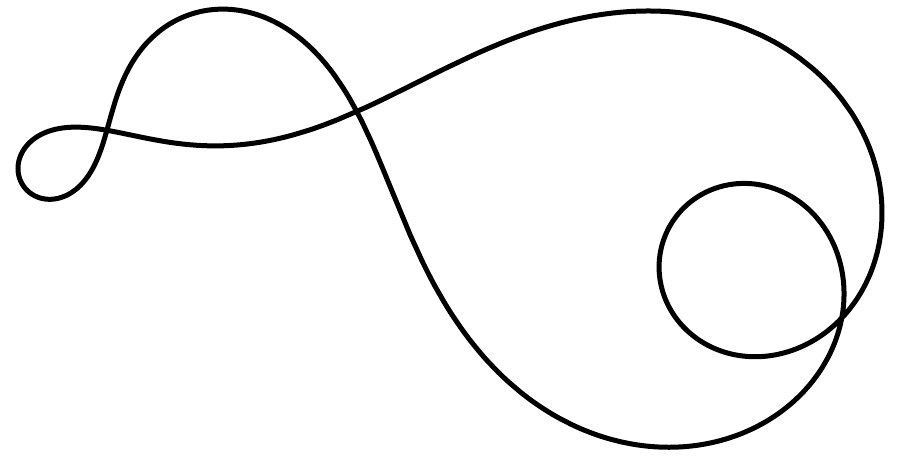}\\
	\includegraphics[width=1.85cm]{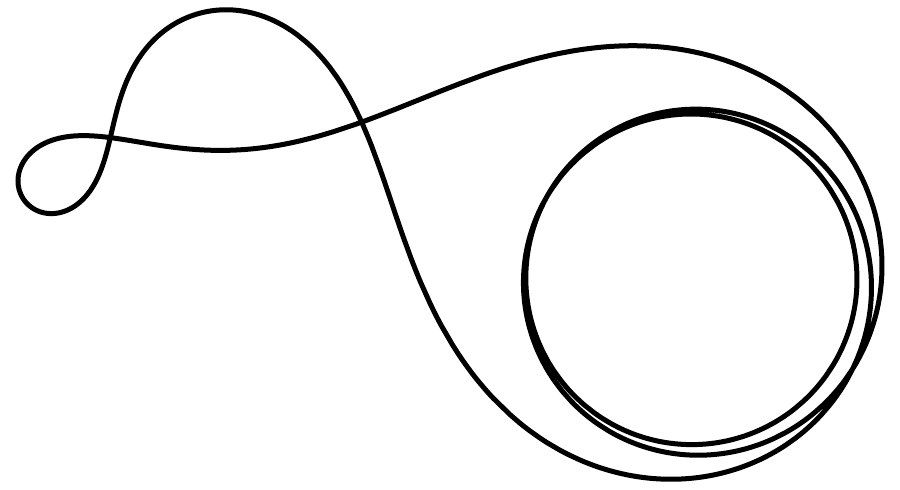}
	\includegraphics[width=1.85cm]{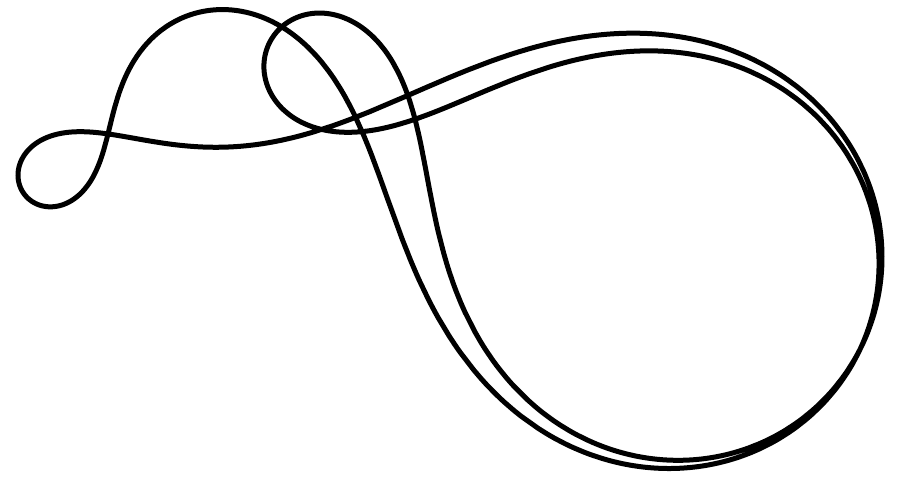}
	\includegraphics[width=1.85cm]{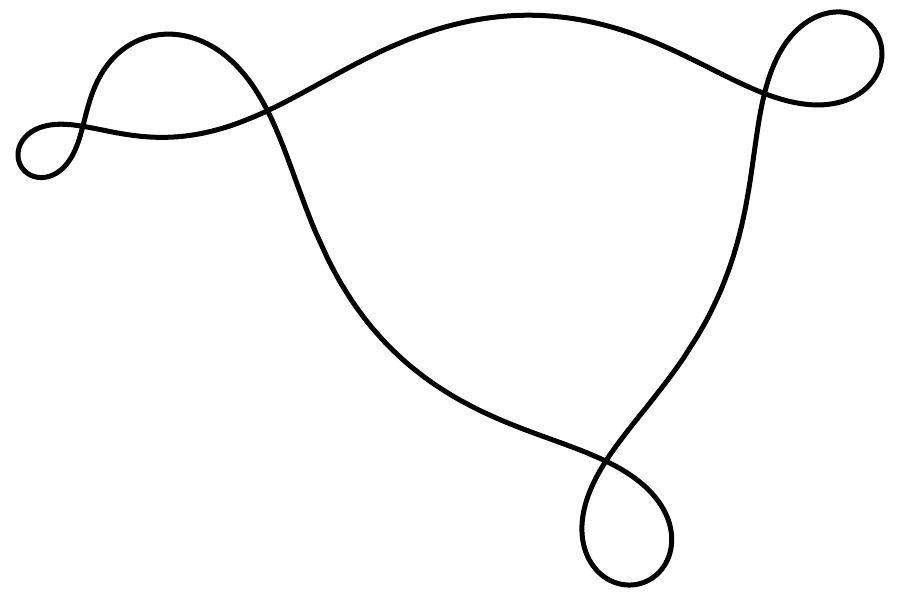}
	\includegraphics[width=1.85cm]{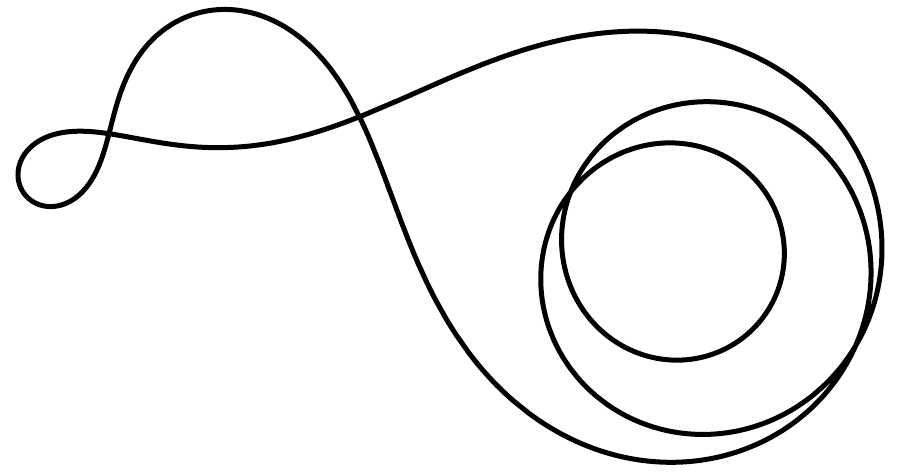}
	\includegraphics[width=1.85cm]{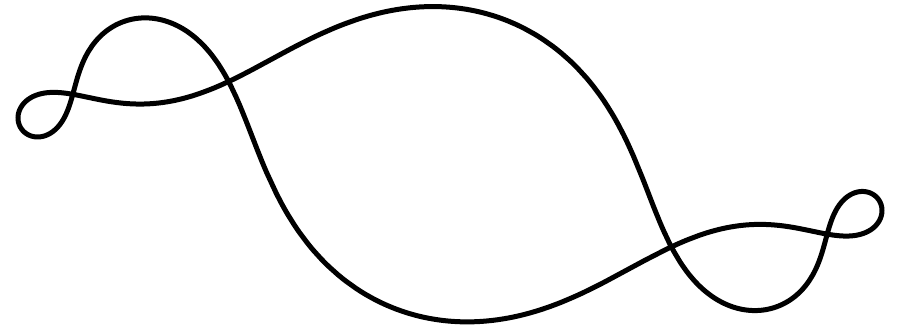}
	\includegraphics[width=1.85cm]{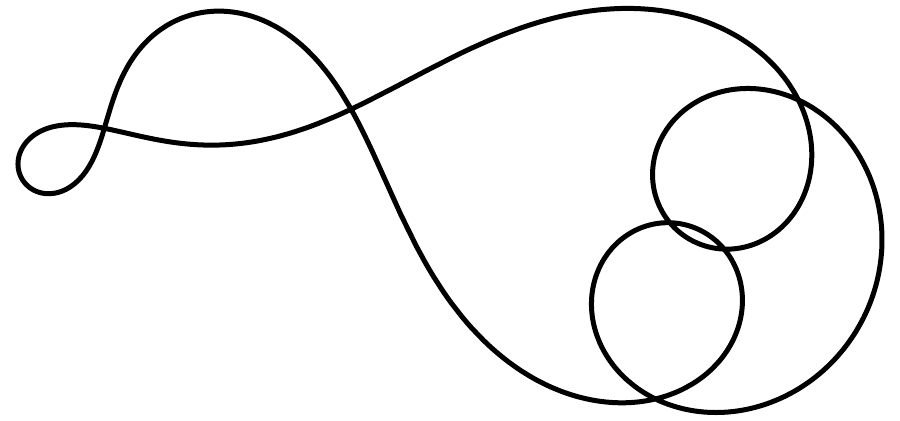}
\caption{ \label{fig:double-circletons1}
Double-circletons: In the first row a $(1,2;3)$-double-circleton. In the second row from left to right:  $(1,2;4),\,(1,3;4)$ and a $(2,3;4)$-double-circletons.  In the third row from left to right: $(1,2;5),\,(1,3;5),\,(1,4;5),\,(2,3;5),\,(2,4;5) $ and $(3,4;5)$-double-circletons.}
\end{figure}
\begin{figure}
  	\includegraphics[width=2.5cm]{circleton1323}
	\includegraphics[width=2.5cm]{circleton1434}
	\includegraphics[width=2.5cm]{circleton1545}
	\includegraphics[width=2.5cm]{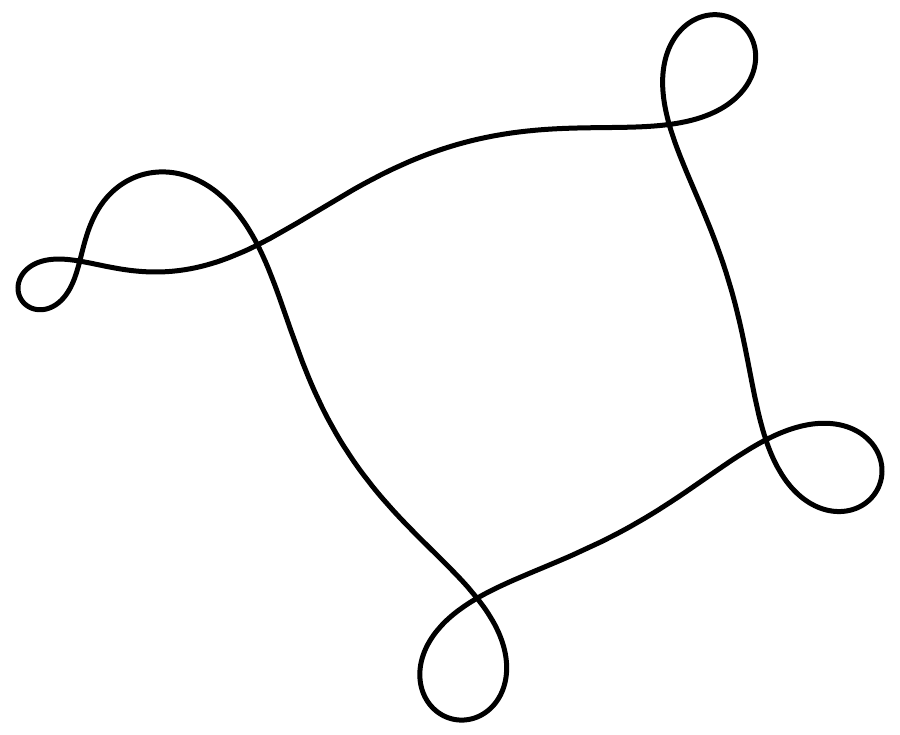}
	\includegraphics[width=2.5cm]{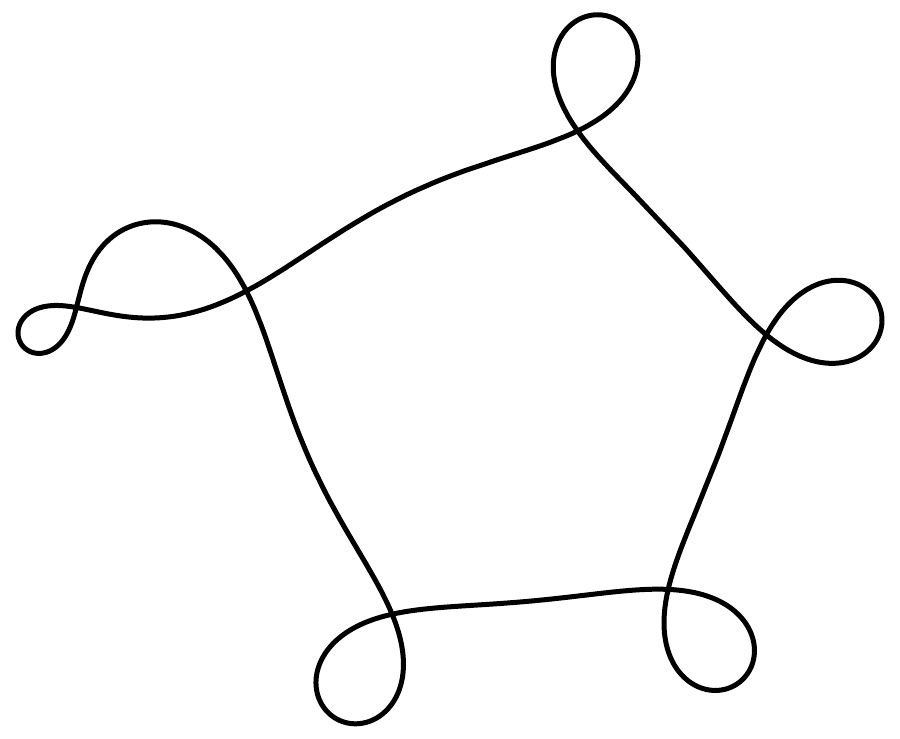}
	\includegraphics[width=2.5cm]{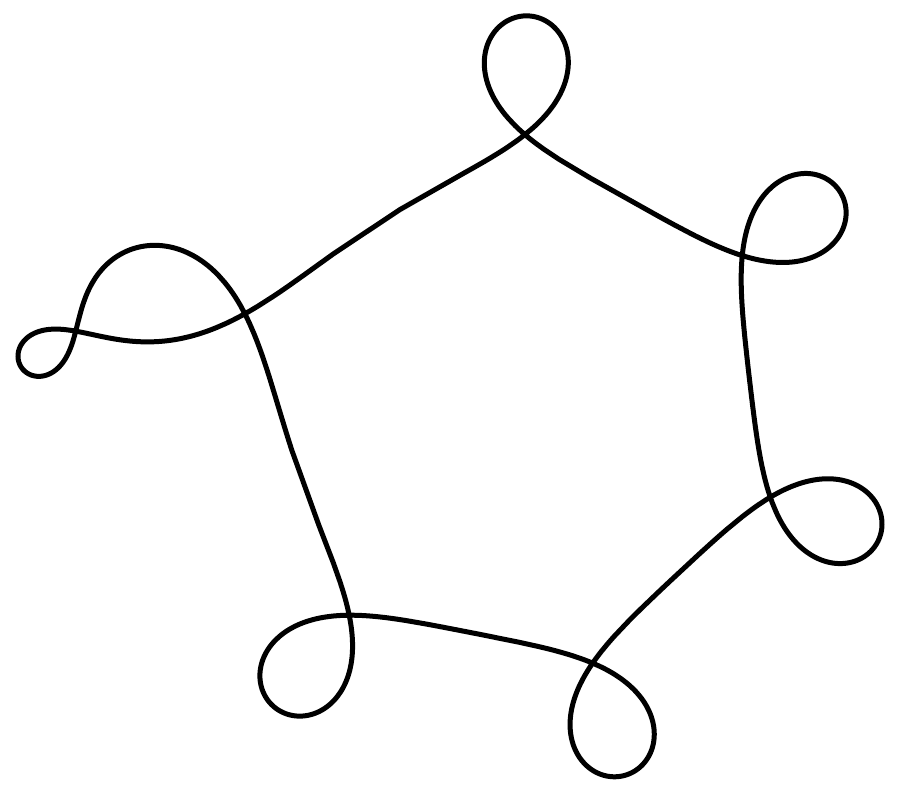}
\caption{ \label{fig:double-circletons2} 
Some $(1,\omega-1;\omega)$ double-circletons for $3 \leq \omega \leq 8$. Dressing twice adds a little loop to the once-dressed curve.}
\end{figure}
%
\bibliographystyle{amsplain}

\def\cydot{\leavevmode\raise.4ex\hbox{.}} \def\cprime{$'$}
\providecommand{\bysame}{\leavevmode\hbox to3em{\hrulefill}\thinspace}
\providecommand{\MR}{\relax\ifhmode\unskip\space\fi MR }
\providecommand{\MRhref}[2]{%
  \href{http://www.ams.org/mathscinet-getitem?mr=#1}{#2}
}
\providecommand{\href}[2]{#2}

\end{document}